\documentclass[authoryear,12pt]{article}
\usepackage[english]{babel}
\usepackage{url,amsfonts,epsfig}
\usepackage{bm}
\usepackage{algorithm}
\usepackage{amsfonts}
\usepackage{stmaryrd}
\usepackage{mathrsfs}
\usepackage{amsmath}
\usepackage{enumerate}
\usepackage{mathrsfs}
\usepackage{mathtools}
\usepackage{graphicx}
\usepackage{amsthm}
\usepackage{natbib}
\usepackage{subfigure}
\usepackage{tikz}
\usepackage{vmargin}
\usepackage{amssymb}
\usepackage{lscape}
\setpapersize{A4}
\newcommand{\GIG}{\ensuremath{\mathrm{GIG}}}
\newcommand{\HD}{\ensuremath{\mathrm{HD}}}
\newcommand{\ED}{\ensuremath{\mathrm{ED}}}
\newcommand{\ST}{\ensuremath{\mathrm{ST}}}
\newcommand{\SC}{\ensuremath{\mathrm{SC}}}
\newcommand{\SN}{\ensuremath{\mathrm{SN}}}
\newcommand{\CSN}{\ensuremath{\mathrm{CSN}}}
\newcommand{\GH}{\ensuremath{\mathrm{GH}}}
\newcommand{\CGH}{\ensuremath{\mathrm{CGH}}}
\newcommand{\CST}{\ensuremath{\mathrm{CST}}}
\newcommand{\CSC}{\ensuremath{\mathrm{CSC}}}
\newcommand{\NIG}{\ensuremath{\mathrm{NIG}}}
\newcommand{\St}{\ensuremath{\mathrm{St}}}
\newcommand{\CNIG}{\ensuremath{\mathrm{CNIG}}}
\newcommand{\CSt}{\ensuremath{\mathrm{CSt}}}
\newcommand{\diag}{\ensuremath{\mathrm{diag}}}
\newtheorem{theorem}{Theorem}[section]
\newtheorem{example}[theorem]{Example}
\newtheorem{corollary}[theorem]{Corollary}

\renewcommand{\vec}[1]{\mathbf{#1}}

\setmarginsrb{23mm}{23mm}{23mm}{30mm}{0pt}{0mm}{0pt}{0mm}
\begin{document}
\title{On the Computation of Multivariate Scenario Sets for the Skew-$t$ and Generalized Hyperbolic Families}
\author{Emanuele Giorgi$^{1,2}$, Alexander J. McNeil$^{3,4}$}

\maketitle
\begin{abstract}
\text{ } We examine the problem of computing multivariate scenarios
sets for skewed distributions. Our interest is motivated by the
potential use of such sets in the ‘stress testing’ of insurance
companies and banks whose solvency is dependent on changes in a set of
financial ‘risk factors’. We define multivariate scenario sets based
on the notion of half-space depth (HD) and also introduce the notion
of expectile depth (ED) where half-spaces are defined by expectiles
rather than quantiles. We then use the HD and ED functions to define
convex scenario sets that generalize the concepts of quantile and
expectile to higher dimensions. In the case of elliptical
distributions these sets coincide with the regions encompassed by the
contours of the density function. In the context of multivariate
skewed distributions, the equivalence of depth contours and density
contours does not hold in general. We consider two parametric families
that account for skewness and heavy tails: the generalized hyperbolic
and the skew-$t$ distributions.  By making use of a canonical form
representation, where skewness is completely absorbed by one
component, we show that the HD contours of these distributions are
‘near-elliptical’ and, in the case of the skew-Cauchy distribution, we
prove that the HD contours are exactly elliptical. We propose a
measure of multivariate skewness as a deviation from angular
symmetry and show that it can explain the quality of the elliptical
approximation for the HD contours.  \\
\text{ }
\end{abstract}
\begin{small}
1. Lancaster Medical School, Lancaster University, Lancaster, UK\\
2. Institute of Infection and Global Health, University of Liverpool, Liverpool, UK\\
3. Department of Actuarial Mathematics and Statistics, Heriot-Watt University, Edinburgh, UK\\
4. Maxwell Institute for Mathematical Sciences, Edinburgh, UK
\end{small}

\pagebreak

\section{Introduction}
\label{sec:intro}

While the topic of this paper is of independent computational statistical interest, the original motivation for studying these issues comes from applications in financial risk management.

Let $\vec{X}$ be a random vector representing changes in a set of
so-called financial risk factors, such as equity indexes, interest
rates, foreign exchange rates, etc. These risk factors impact the
value of a financial portfolio and lead to a random loss given by
$L=\ell(\vec{X})$. The portfolio might be a derivatives desk at a bank
or a product book (e.g. annuity book) of an insurer.

We will assume that: (1) we have data that permit the statistical
estimation of a model for $\vec{X}$; (2) the function $\ell$ is known
to us. That is, for any value $\vec{x}$ we are able to compute the
resulting loss $\ell(\vec{x})$. The function $\ell$ contains
information about the size of the positions in the portfolio and
encapsulates the valuation formulas necessary to quantify the effect of changes in the risk factors on the values of the positions.

A first question of possible interest is: how do we construct a scenario set $S$ based on the probability distribution of $\vec{X}$ that includes plausible scenarios?

In our opinion there are advantages to using scenario sets that are
based on the idea of half-space depth rather than sets that are based
on density. In the presence of multivariate skewness, density sets
tend to be dragged towards the shorter tails of the distribution and
exclude too many extreme scenarios in the outer tail. 

An example is
given in Figure~\ref{fig:motivating_data}. We have plotted typical one-year
changes in yield for a 3-year government bond and a 10-year
government bond. These are the kinds of risk factors that would be
considered in quantifying, for example, the change in the value of a
portfolio of government bonds, or a portfolio of annuity liabilities. A bivariate normal inverse Gaussian
(\NIG) distribution has been fitted to the data (see Example \ref{example:motivating_data}) and, on the basis of the
fitted model, density contours have been plotted and lines are shown that
divide the plane into two half-spaces with probabilities $\alpha=0.005$ and
$1-\alpha=0.995$. The set formed by the intersection of closed half-spaces
with probability $1-\alpha$ is denoted $Q_\alpha$. Points on the
boundary $\delta Q_\alpha$ are said to have \textit{depth} $\alpha$ so
that $Q_\alpha$ is the set of points with depth at least $\alpha$. 

\begin{figure}[htbp]
\begin{center}
\includegraphics[scale=0.9]{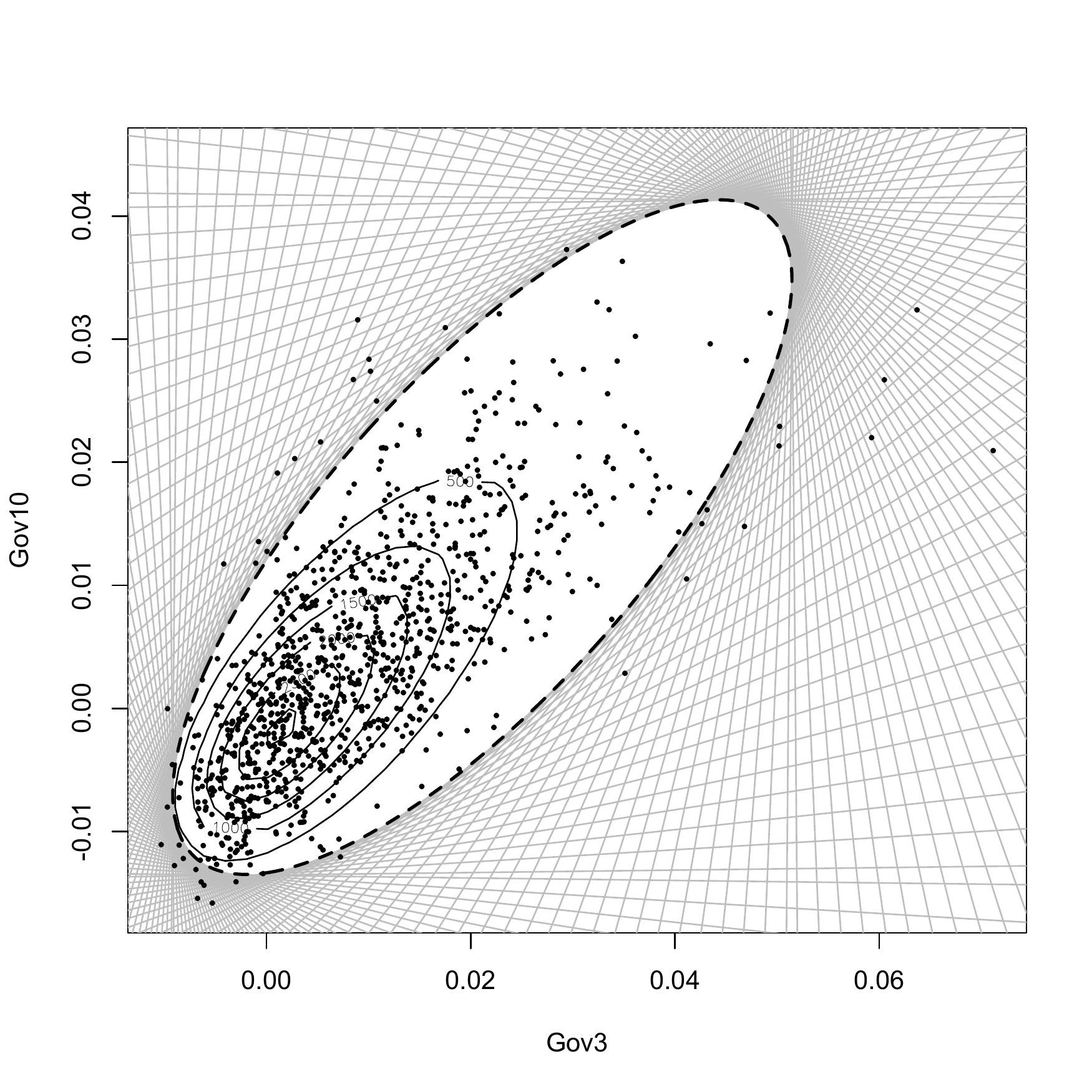}
\caption{Picture shows change in yields for 3-year and 10-year
  government bond over 1 year. A bivariate NIG distribution has been
  fitted and corresponding density contours. The dashed line
  corresponds to the approximating ellipsoid for the $\alpha$-depth set
  $Q_{\alpha}$ for $\alpha=0.005$; the grey lines show the boundaries of half
  spaces with probabilities $\alpha$ and $(1-\alpha)$.\label{fig:motivating_data}}
\end{center}
\end{figure}

% There are also elegant theoretical links between sets based on half-space
% depth and risk measures used in risk management in the case where the
% loss function $\ell$ is linear~(\cite{bib:mcneil-smith-12}), which are
% discussed later in this paper. 

If we construct a scenario set $Q_\alpha$ based on depth we want to be able to say
whether $\vec{x} \in Q_\alpha$ for some arbitrary point $\vec{x}$. It is often the case that a regulator or manager asks the risk
modeller to consider a particular extreme scenario and work out how
costly it might be. In order to weigh the importance or plausibility of this scenario
we would like to know its depth $\alpha$, i.e.~the largest value of
$\alpha$ for which $\vec{x} \in Q_\alpha$.

Suppose we are given a whole series of scenarios to consider and for each scenario we calculate the associated loss. Let $R
= \{\vec{x} : \ell(\vec{x}) > \ell_0\}$ for some $\ell_0$ denote a
ruin set, i.e.~a set of scenarios that lead to an unacceptably large
loss. We would like to identify the ruin scenario that is most plausible, in
the sense that it has the maximum depth $\alpha$.
This is known as the reverse stress testing problem.

In this paper we will consider the properties of $Q_\alpha$ and
related scenario sets when the distribution of
$\vec{X}$ is
in either the skew-$t$ (\ST) or
generalised hyperbolic (\GH) family. These are flexible families
of skewed and potentially heavy-tailed distributions that are useful
for modelling financial risk-factor changes. The NIG distribution
fitted to the data in
Figure~\ref{fig:motivating_data} belongs to the latter family and we
see that, in this case, the set $\delta Q_\alpha$ is very
close to (but not exactly) an ellipse. This ``near ellipsoidal''
behaviour is true of other distributions in these families.

The paper is structured as follows. In Section \ref{sec:hd_ssets}, we
define multivariate scenario sets based on half-space depth (\HD)
and briefly describe some important properties of the half-space
median that are relevant for the purposes of this paper. In Section
\ref{sec:exp_depth} we introduce the related notion of expectile depth (\ED)
which generalizes the univariate concept of expectile to higher dimensions. 

In Section \ref{sec:depth_skewedd} we examine the problem of computing
depth sets based on \HD \ and \ED \ for the \ST \ (Section
\ref{subsec:skewt}) and \GH \ (Section \ref{subsec:gh}) families of
distributions. We show that the computation of depth sets can be simplified
by making use of a canonical form representation, where skewness is
completely absorbed by one component. We also prove that the \HD \
contours for the skew-Cauchy (\SC) distribution are elliptical, giving
a further simplification in the computation. 

We provide an algorithm for the construction of approximating
ellipsoids to depth sets based on \HD \ and investigate the quality of
such an approximation. In Section \ref{subsec:relation_rs_eds} we
examine the relationship between ellipsoidal depth sets and angular
symmetry. We propose a multivariate measure of skewness as a deviation
from angular symmetry and show that this measure can explain the
quality of the ellipsoidal approximation. Numerical results, in
Section \ref{subsec:prob_misclass}, show its concordance with the
probability of misclassification. Although the probability of
misclassification is a direct and more interpretable measure of the
error we incur when using ellipsoidal approximations, it is much more
difficult to compute.  We conclude with a discussion in Section \ref{sec:discussion}. 

\section{Half-Space Depth}
\label{sec:hd_ssets}

\subsection{Half-Space Depth and its Relation to Quantiles}
\label{subsec:hd_quantiles}
Let $(\Omega, \mathscr{F}, P)$ be a probability space and $\vec{X} : \Omega \rightarrow \mathbb{R}^d$ be a given random vector. For any $\vec{y} \in \mathbb{R}^d$ and any directional vector $\vec{u} \in \mathbb{R}^{d} \setminus \{0\}$, let 
$$H_{\vec{y},\vec{u}} = \{\vec{x} \in \mathbb{R}^d : \vec{u}^\top\vec{\vec{x}} \leq \vec{u}^\top\vec{y}\}$$
denote the closed half-space bounded by the hyperplane through $\vec{y}$. We write the probability that $\vec{X}$ lies in the half-space $H_{\vec{y}, \vec{u}}$ as
$$P_{\vec{X}}(H_{\vec{y}, \vec{u}}) = P(\vec{u}^\top\vec{X} \leq \vec{u}^\top\vec{y}).$$
The half-space depth of the point $\vec{x} \in \mathbb{R}^d$  with
respect to the probability distribution of $\vec{X}$ is given by
\begin{equation}
\label{eq:hd}
\HD_{\vec{X}}(\vec{x}) = \inf_{\|u\| = 1}P_{\vec{X}}(H_{\vec{x},\vec{u}}),
\end{equation}
where $\|\cdot\|$ is the Euclidean norm.  We note that \HD \ is an
affine invariant measure, meaning that if $A\in \mathbb{R}^{d \times
  d}$ is a non-singular matrix and $\vec{b}\in\mathbb{R}^d$ is a
vector, then $\HD_{A\vec{X}+\vec{b}}(A\vec{x}+\vec{b}) = \HD_{\vec{X}}(\vec{x})$.\par 

Let $\alpha \in (0,0.5]$ be a probability value. The main definition
of a scenario set that we use is
\begin{eqnarray}
\label{eq:depth_set}
Q_{\alpha} = \bigcap \{H_{\vec{y},\vec{u}} : P_{\vec{X}}(H_{\vec{y},\vec{u}}) \geq 1-\alpha\}
\end{eqnarray}
which is the intersection of all closed half-spaces with probability
at least $(1-\alpha)$. Sets of this kind are considered by many authors
including~\cite{bib:masse-theodorescu-94},~\cite{bib:rousseeuw-ruts-99}
and~\cite{bib:mcneil-smith-12}. The construction is sometimes
referred to as half-space trimming.\par

For $\theta \in (0,1)$ we also define a $\theta$-quantile function on $\mathbb{R}^d \setminus \{0\}$ by writing $q_{\theta}(\vec{u})$ for the $\theta$-quantile of $\vec{u}^\top\vec{X}$. Then, \eqref{eq:depth_set} can be expressed in terms of $q_{\theta}(\vec{u})$ by
$$
Q_{\alpha} = \left\{\vec{x} : \vec{u}^\top\vec{x} \leq q_{1-\alpha}(\vec{u}), \forall \vec{u}\right\}.
$$
\par

%\begin{assumption}
We will make the assumption that $\vec{X}$ has a
  strictly positive probability density on $\mathbb{R}^d$. 
%\end{assumption}
This
  assumption is satisfied by the distributions that interest us in
  this paper and allows us to pass easily between the concepts of
  quantiles and depth. Under this assumption, we have
\begin{displaymath}
   \left\{\vec{x} : \vec{u}^\top\vec{x} \leq
     q_{1-\alpha}(\vec{u}) \right\} =
  \left\{\vec{x} : P_{\vec{X}}(H_{\vec{x},\vec{u}}) \leq 1-\alpha\right\} 
 \end{displaymath}
from which it can be easily deduced that
\begin{equation}
  \label{eq:Qalpha-depth}
  Q_\alpha = \left\{\vec{x} : \HD_{\vec{X}}(\vec{x}) \geq \alpha \right\}.
\end{equation}
We thus refer to $Q_\alpha$ as a depth set and we refer to $\partial
Q_{\alpha}$, the boundary of
$Q_{\alpha}$, as the $\alpha$ depth contour; the contour
consists of the points with depth exactly equal to $\alpha$. 

In Figure~\ref{fig:motivating_data} we show an example of the half-space trimming
construction for $\alpha=0.005$ (the mesh of grey straight lines) as
well as the depth contour $\partial Q_{0.005}$ (the dashed curve).

\subsection{The half-space median and angular symmetry}
\label{subsec:hd_median}
The function $\HD_{\vec{X}}(\vec{x})$ can be used to define an affine
equivariant median known as the half-space median (or Tukey median) of
$\vec{X}$. This is the set of maximal $\HD$ given by
\begin{equation}
\label{eq:hd_median}
\boldsymbol{\beta}_{\vec{X}} = \arg \max_{\vec{x} \in \mathbb{R}^d} \HD_{\vec{X}}(\vec{x}).
\end{equation}
By affine equivariant we mean that, if $A\in \mathbb{R}^{d \times d}$
is a non-singular matrix and $\vec{b} \in \mathbb{R}^d$ is a vector,
then
$\boldsymbol{\beta}_{A\vec{X}+\vec{b}} = A
\boldsymbol{\beta}_{\vec{X}} + \vec{b}$.\par
The half-space median is, in general, not unique unless $\vec{X}$ is symmetric
according to some notion of multivariate symmetry; see
\citet*{serfling2006} for a survey of multivariate concepts of
symmetry. More general conditions for uniqueness are given in
\citet*{Small1987} who shows that a sufficient condition for uniqueness
in the bivariate case is the strict positivity of the density
function. The least restrictive definition of multivariate symmetry,
under which the half-space median is unique, is \emph{angular
  symmetry}. The random variable $\vec{X}$ is angularly symmetric about a point
$\boldsymbol{\eta}$ if
\begin{equation}
\label{eq:angular_sym}
\frac{\vec{X}-\boldsymbol{\eta}}{\|\vec{X}-\boldsymbol{\eta}\|} \overset{d}{=} -\frac{\vec{X}-\boldsymbol{\eta}}{\|\vec{X}-\boldsymbol{\eta}\|},
\end{equation}
where ``$\overset{d}{=}$'' indicates equality in
distribution. \citet*{dutta2011} show that $\HD_{\vec{X}}(\boldsymbol{\eta}) =
1/2$ if and only if $\boldsymbol{\eta}$ is the center of angular
symmetry. Thus, if $\boldsymbol{\eta}$ is the centre of angular
symmetry, then $\boldsymbol{\beta}_{\vec{X}}=\boldsymbol{\eta}$. This property is used in Section
\ref{subsec:relation_rs_eds} to define two different measures of multivariate skewness as deviations from angular symmetry. \par
Note that if $\boldsymbol{\beta}_{\vec{X}}$ is the center of angular symmetry,
then $\boldsymbol{\beta}_{\vec{X}}$ also corresponds to the component-wise
median. This is clear, since if $\HD_{\vec{X}}(\boldsymbol{\beta}_{\vec{X}}) =
1/2$ then
$$
P_{\vec{X}}(H_{\boldsymbol{\beta}_{\vec{X}},\vec{u}}) = 1/2,\quad \forall \vec{u} \in \mathbb{R}^d.
$$
If we set $\vec{u} = \vec{e}_{i}$, the $i$th unit vector, we can infer that the $i$th element of $\boldsymbol{\beta}_{\vec{X}}$ is the univariate median of the marginal distribution of the $i$th component of $\vec{X}$.
\section{Expectile Depth}\label{sec:exp_depth}
\subsection{Definitions}

The notion of an expectile was introduced by \citet*{newey1987} as the
solution of an asymmetric least squares regression problem, analagous to
quantile regression. Given an integrable random variable $Y$ in
$\mathbb{R}$ and $\theta \in (0,1)$, the $\theta$-expectile of the distribution function
$F_Y$ of $Y$ is the unique solution $y$ of the
equation
\begin{equation}
  \label{eq:expectile-definition}
 \theta \mathbb{E}((Y-y)^+) =  (1-\theta)\mathbb{E}((Y-y)^-)
\end{equation}
where $x^+ = \max(x,0)$, $x^-=\max(-x,0)$ and
$\mathbb{E}(\cdot)$ is the expectation with respect to the
distribution of $Y$. 
% $$
% \tilde{e}_{\alpha}(X) = \arg \min_{x \in \mathbb{R}} \mathbb{E}[\alpha \max(X-x,0)^2+(1-\alpha)\max(x-X,0)^2],
% $$

It was shown by~\cite{jones1994} that the $\theta$-expectile of $F_Y$ can be
expressed as the $\theta$-quantile of the related distribution function
\begin{displaymath}
  \tilde{F}_Y(y) = \frac{yF_Y(y)-\mu(y)}{2(yF_Y(y)-\mu(y))+\mathbb{E}(Y) -y}
\end{displaymath}
where $\mu(y):=\int_{-\infty}^y x d F_Y(x)$
is the lower partial moment of $F_Y$. For any random variable $Y$ the
distribution function $\tilde{F}_Y(y)$ is continuous and strictly increasing on
its support, implying that the $\theta$-expectile is uniquely defined for all
$\theta$ in $(0,1)$.

In our application, for a fixed random vector $\vec{X}\in
\mathbb{R}^d$, we define the $\theta$-expectile function
$e_{\theta}(\vec{u})$ to be the $\theta$-expectile of the distribution
$\vec{u}^\top\vec{X}$ for $\vec{u} \in \mathbb{R}^d$ and, for $\alpha
\in (0,0.5]$ we 
% by
% $$
% e_{\alpha}(\vec{u}) = \tilde{e}_{\alpha}(\vec{u}^\top\vec{X}), \vec{u} \in \mathbb{R}^d
% $$
% and 
consider a scenario set of the form
$$
E_{\alpha} = \left\{\vec{x} : \vec{u}^\top\vec{x} \leq e_{1-\alpha}(\vec{u}), \forall \vec{u}\right\}.
$$
Let 
$$
\ED_{\vec{X}}(\vec{x}) = \inf_{\|\vec{u}\|=1}
\tilde{P}_{\vec{X}}(H_{\vec{x},\vec{u}})
$$
denote the smallest probability of a half-space
$H_{\vec{x},\vec{u}}$ when probabilities are calculated according to $\tilde{P}_{\vec{X}}(H_{\vec{x},\vec{u}}) = \tilde{F}_{\vec{u}^\top\vec{X}}(\vec{u}^\top\vec{x})$. We refer to
$\ED_{\vec{X}}(\vec{x})$ as the expectile depth of $\vec{x}$ with respect to the
distribution of $\vec{X}$ and note that it is also an affine invariant
measure. The
scenario set may also be expressed as
$$
E_{\alpha} = \left\{\vec{x} : \ED_{\vec{X}}(\vec{x}) \geq \alpha\right\}.
$$

With obvious notation, we use $\partial E_{\alpha}$ to indicate the
boundary of $E_{\alpha}$, and refer to this as the $\alpha$-expectile
depth contour.

\subsection{Properties of expectile depth}
There are both practical and theoretical reasons for considering
expectile depth as an alternative to standard (quantile) depth.
On the one hand the expectile can simply be viewed as a kind of
generalized quantile;
see~\citet{bib:bellini-klar-mueller-gianin-13}. Using the techniques
of \citet{bib:mcneil-smith-12}
it is straightforward
to show that,
 when $\vec{X}$ has an
elliptical distribution, the set
$E_\alpha$ is an ellipsoidal set like $Q_\alpha$, with axis lengths in
identical proportions. However, for general distributions expectile
depth sets will have different shapes to (quantile) depth sets. 
% For
% Pareto-tailed distributions with tail index $1 < \beta
% < 2$ (infinite variance distributions) these authors show that $\theta$-expectiles are larger than
% $\theta$-quantiles for large enough $\theta$.  Thus, for these very
% heavy-tailed distributions, expectile depth contours give an alternative description of the shape
% that is even more weighted towards the extremes of
% the distribution.

At a more theoretical level, both the quantile function $q_\theta(\vec{u})$ and the expectile
function $e_\theta(\vec{u})$ are positive-homogeneous functions on
$\mathbb{R}^d$, meaning that they are functions $r: \mathbb{R}^d \to
\mathbb{R}$ satisfying $r(k\bm{u}) = k r(\bm{u})$ for $k >0$.

A fundamental result in convex analysis can be used to show that a
positive-homogeneous function $r$ has a representation as 
the so-called support function
 \begin{align}\label{eq:stress-test-representation}
 r(\bm{u}) &= \sup \{\bm{u}^\top \bm{x} : \bm{x} \in S \} \,
 \end{align}
of the convex set
\begin{equation}\label{eq:1}
S = \{ \mathbf{x} \in \mathbb{R}^d \colon   \bm{v}^\top \mathbf{x} \leq  r(\bm{v})  \text{ for all } \bm{v}\in \mathbb{R}^d \}\,
\end{equation}
if and only if the function $r$ is also subadditive on
$\mathbb{R}^d$; see~\citet{bib:rockafellar-70}
or~\citet{bib:mcneil-smith-12} for technical details and recall that a
subadditive function satisfies $r(\bm{u}_1+\bm{u}_2) \leq r(\bm{u}_1)
+ r(\bm{u}_2)$ for all $\bm{u}_1$ and $\bm{u}_2$ in $\mathbb{R}^d$.

The set $S$ in~\eqref{eq:1} is $Q_\alpha$ when $r =
q_{1-\alpha}$ and it is $E_\alpha$ when $r =
e_{1-\alpha}$. However, the quantile function $q_\theta$ is not subadditive in
general, but only for certain underlying random vectors
$\vec{X}$ and certain values of $\theta$. For example, for elliptically distributed random
vectors $q_\theta$ is subadditive for $\theta>0.5$. But when
$\vec{X}=(X_1,X_2)$ is a vector comprising two independent standard
exponential random variables, \citet{bib:mcneil-smith-12} show that
$q_\theta$ is not subadditive for $\theta =0.72$. In contrast,
the expectile function $e_\theta$ is subadditive for any random vector
$\vec{X}$ with finite mean and $\theta>0.5$.

Let $\alpha$ satisfy $0 < \alpha <0.5$ and assume that $E_\alpha$ and
$Q_\alpha$ are non-empty. The implication
of~\eqref{eq:stress-test-representation} is that for any $\bm{u}\in
\mathbb{R}^d$ there exists a scenario $\bm{x}(\bm{u}) \in \partial E_\alpha$, i.e.~a scenario on the
boundary of the expectile depth set,
such that $e_{1-\alpha}(\bm{u}) = \bm{u}^{\top}
\bm{x}(\bm{u})$. However, there may exist values $\bm{u}\in \mathbb{R}^d$
such that $q_{1-\alpha}(\bm{u}) > \bm{u}^{\top}
\bm{x}$ for all $\bm{x} \in Q_\alpha$. In this sense $\partial E_\alpha$ is 
more satisfactory as a multivariate analogue of the expectile than is
$\partial Q_\alpha$ as a multivariate analogue of the quantile.

We note that scenario sets of the form~\eqref{eq:1} can be based on
other positive-homogeneous and subadditive
functions. In~\citet{bib:mcneil-smith-12} a scenario set based on the
function $es_\theta(\bm{u}) = E(\bm{u}^\top\vec{X} \mid
\bm{u}^\top\vec{X} \geq q_\theta(\bm{u}))$ is proposed; this is
related to the so-called expected shortfall risk measure.

% For (5) note that part (2) implies that the risk measure function $r_\varrho(\bm{\lambda})$ takes the form $r_\varrho(\bm{\lambda}) = \|A'\bm{\lambda}\|\varrho(Y)+\bm{\lambda}'\bm{\mu}$ so that the set $S_\varrho$ in~(\ref{eq:stress-test-representation}) is
% \begin{align*}
%   S_\varrho &= \left\{ \bm{x} \in \RR^d \;:\; \bm{u}' \bm{x} \leq  \bm{u}' \bm{\mu} + \|A'\bm{u}\|\;\varrho(Y) , \forall \bm{u}\in \RR^d \right\} \\
% &=  \left\{\bm{x} \in \RR^d  \;:\; \bm{u}' A A^{-1}( \bm{x} - \bm{\mu}) \leq  \|A' \bm{u}\|\; \varrho(Y) , \forall \bm{u}\in \RR^d \right\} \\
% &=  \left\{ \bm{x} \in \RR^d \;:\; \bm{v}' \frac{A^{-1}( \bm{x} - \bm{\mu})}{\varrho(Y)} \leq  \| \bm{v}\|  , \forall \bm{v}\in \RR^d \right\}
% \end{align*}
%  where the last line follows because $\RR^d =\{ A^\prime \bm{u} \;:\; \bm{u} \in \RR^d \}$. By observing that the Euclidean unit ball $\{\mathbf{y} \;:\; \bm{y}' \bm{y} \leq 1 \}$ can be written as the set $\{\bm{y} \;:\; \bm{v}' \bm{y} \leq \|\bm{v}\|, \forall \bm{v} \}$ we conclude that, for $\bm{x}\in S_\varrho$, the vectors $\bm{y}= A^{-1} (\bm{x}-\bm{\mu})/\varrho(Y)$ describe the unit ball and therefore
%  \begin{displaymath}
%    S_\varrho = \{ \bm{x} \in \RR^d  \; : \; (\bm{x} -\bm{\mu})' \Sigma^{-1}(\bm{x} -\bm{\mu}) \leq \varrho(Y)^2\}\;.
%  \end{displaymath}

\section{Depth Sets for Skewed Distributions}
\label{sec:depth_skewedd}

We now consider two families of multivariate skewed distributions, both
of which have a canonical form, obtained by an affine transformation,
 in which all of the skewness is
absorbed by one of the marginal distributions. In view of the affine
invariance of \HD \ and \ED, it suffices to be able to calculate these
quantities for random vectors in their canonical form.

\subsection{Skew-$t$ distribution}
\label{subsec:skewt}
The skew-$t$ (\ST) distribution is a flexible model for
skewed and heavy-tailed multivariate data \citep*{azzalini2003,
  azzalini_genton2008}. Let $\boldsymbol{\xi}, \boldsymbol{\gamma} \in
\mathbb{R}^{d}$ and $\nu \in \mathbb{R}^{+}$ denote the parameters of
location, skewness and the degrees of freedom; let
$\Omega\in\mathbb{R}^{d \times d}$ be a symmetric, positive-definite
dispersion matrix. A random variable $\vec{X}$ in $\mathbb{R}^{d}$ is distributed
according to a $\ST_{d}(\boldsymbol{\xi}, \Omega, \boldsymbol{\gamma},
\nu)$ distribution if it has density function
\begin{equation*}
f_{\ST_{d}}(\vec{x}) = 2t_{d}(\vec{x};\nu)T_{1}\left(\boldsymbol{\gamma}^{\top}\omega^{-1}(\vec{x}-\boldsymbol{\xi})
             \left(\frac{\nu+d}{Q_{\vec{x}}+\nu}\right)^{1/2};\nu+d\right)\text{, }\vec{x} \in \mathbb{R}^{d},
\end{equation*} 
where 
\begin{eqnarray*}
 t_{d}(\vec{x};\nu) = \frac{\Gamma((\nu+d)/2)}{|\Omega|^{1/2}(\pi\nu)^{d/2}\Gamma(\nu/2)}(1+Q_{\vec{x}}/\nu)^{(\nu+d)/2}\text{, }Q_{\vec{x}} = (\vec{x}-\boldsymbol{\xi})^{\top}\Omega^{-1}(\vec{x}-\boldsymbol{\xi}),
\end{eqnarray*}
\begin{equation*}
 \omega = \diag (\omega_{1},\ldots,\omega_{d}) = \diag (\omega_{11},\ldots,\omega_{dd})^{1/2}
\end{equation*}
and $T_{1}(\cdot;\nu)$ denotes the univariate Student $t$
distribution function with $\nu$ degrees of freedom. For later use, we
also define the following correlation matrix
$$
\overline{\Omega} = \omega^{-1}\Omega\omega^{-1}.
$$
% and the vector
% $$
% \boldsymbol{\delta} = \frac{\overline{\Omega}\boldsymbol{\gamma}}{(1+\boldsymbol{\gamma}^{\top}\overline{\Omega}\boldsymbol{\gamma})^{1/2}}
% $$
% whose elements lie in the interval $(-1,1)$. 

Skewness and tails heaviness are regulated by the parameters
$\boldsymbol{\gamma}$ and $\nu$, respectively; these two parameters
jointly characterize the shape of the distribution. If
$\boldsymbol{\gamma}=0$ the Student $t$ distribution is recovered; if $\nu \rightarrow \infty$ we obtain the skew-normal (\SN) distribution \citep*{azzalini1985,azzalini1999};  if $\nu \rightarrow \infty$ and $\boldsymbol{\gamma}=0$ we obtain the multivariate normal distribution. \par
The next result, which follows from the linear transformation result
given in Appendix~\ref{subsec:lf_st}, introduces the canonical form of
the \ST \ distribution (\CST). The importance of the canonical form in
summarizing important features related to the location, skewness and kurtosis of the \ST \ family is discussed in an unpublished paper by \citet*{capitanio2012}. 

\begin{theorem}
\label{theorem:canon_st}
Let $\vec{X} \sim \ST_{d}(\boldsymbol{\xi},\Omega,
\boldsymbol{\gamma}, \nu)$ where  $\Omega = BB^{\top} $ and $B\in
\mathbb{R}^{d\times d}$. Let $\gamma_{*} =
(\boldsymbol{\gamma}^{\top}\overline{\Omega}\boldsymbol{\gamma})^{1/2}$ and
define $\vec{X}^* = P^{\top}B^{-1}(\vec{X}-\boldsymbol{\xi})$, where
$P$ is an orthonormal matrix with first column equal to
$\gamma_{*}^{-1}B^\top\omega^{-1}\boldsymbol{\gamma}$. Then $\vec{X}^* \sim \ST_{d}(\vec{0},I_{d},\gamma_{*}\vec{e}_{1},\nu)$.
\end{theorem}
If $\vec{X}$ is in canonical form, we simply write $\vec{X} \sim
\CST_{d}(\gamma,\nu)$, where $\gamma$ is the scalar parameter of
skewness; if $\nu=1$, the case of the skew-Cauchy (\SC) distribution,
then $\vec{X} \sim \CSC_{d}(\gamma)$; if $\nu=\infty$, then $\vec{X}
\sim \CSN_{d}(\gamma)$. \par
Without loss of generality we have assumed that for a random vector $\vec{X}$ in canonical form all of the asymmetry
is absorbed in the first component and we have $X_i \overset{d}{=}
-X_i$ for $i\neq 1$.
We now show that the \HD \ and \ED \ contours of the \CST \
distribution are symmetric with respect to the axis of the unique
asymmetric component. This property is useful in the construction of scenario sets based on either \HD \ or \ED \ of any \ST \ distribution, as shown later in some examples.

\begin{theorem}
\label{theorem:symmhd}
Let $\vec{x}$ and $\vec{x}'$ denote two points in $\mathbb{R}^{d}$ whose first elements are equal and the others differ in sign at most. If $\vec{X} \sim \CST_{d}(\gamma,\nu)$, then $\HD_{\vec{X}}(\vec{x}) = \HD_{\vec{X}}(\vec{x}')$ and $\ED_{\vec{X}}(\vec{x}) = \ED_{\vec{X}}(\vec{x}')$.
\end{theorem}
\begin{proof}
For a given normalized vector $\vec{u} \in \mathbb{R}^d$, consider the half-space $H_{\vec{x}, \vec{u}}$. \par
Let $J \subseteq \{2,\ldots,d\}$ be a set of indices indicating the
elements of $\vec{x}'$ that have opposite sign with respect to the
corresponding elements of $\vec{x}$. Define the random vector
$\vec{X}'$ so that $X_i' = -X_i$ if $i\in J$ and $X_i' = X_i$ if $i \in
J^C$, where $J^c$ is the complement of $J$; similarly define $\vec{u}'$ so that  $u_i' = -u_i$ if $i\in J$ and $u_i' = u_i$ if $i \in
J^C$. Since $\vec{X}' \overset{d}{=} \vec{X}$ it follows that
\begin{displaymath}
  P(\vec{u}^\top\vec{X} \leq \vec{u}^\top\vec{x}) =
  P(\vec{u}^\top\vec{X}' \leq \vec{u}^\top\vec{x}) =  P(\vec{u}'^\top\vec{X} \leq \vec{u}'^\top\vec{x}') 
\end{displaymath}
and hence that
\begin{equation*}
P_{\vec{X}}(H_{\vec{x},\vec{u}}) = P_{\vec{X}}(H_{\vec{x}',\vec{u}'})\text{ and } \tilde{P}_{\vec{X}}(H_{\vec{x},\vec{u}}) = \tilde{P}_{\vec{X}}(H_{\vec{x}',\vec{u}'}).
\end{equation*}
We obtain $\HD_{\vec{X}} (\vec{x})=\HD_{\vec{X}} (\vec{x}')$ and
$\ED_{\vec{X}} (\vec{x}) = \ED_{\vec{X}} (\vec{x}')$ when we take the infimum over all $||\vec{u}||=1$.
\end{proof}
In the special case of the \SC \ distribution, the computation of \HD
\ contours is further simplified. As shown in the next theorem, the
\HD \ contours of the \CSC \ distribution are circular, and hence
ellipsoidal for the general \SC \ distribution; we use $\sec(\cdot)$ and $\tan(\cdot)$ to indicate the secant and tangent functions, respectively.
\begin{theorem}
\label{theorem:sc_contours}
If $\vec{X} \sim \CSC_{d}(\gamma)$, then
\begin{eqnarray}
\label{eq:sc_hd_contour}
Q_{\alpha} = \left\{\vec{x}: (x_{1} - s(\alpha))^2 + \sum_{i=2}^dx_{1}^2 \leq t(\alpha)^2 \right\}
\end{eqnarray}
where
\begin{equation*}
s(\alpha) = \frac{\gamma}{\sqrt{1+\gamma^{2}}}\sec\left\{\left(\frac{1}{2}-\alpha\right)\pi\right\}\text{ and }t(\alpha) = \tan\left\{\left(\frac{1}{2}-\alpha\right)\pi\right\}, \alpha \in (0,0.5].
\end{equation*}
\end{theorem}
\begin{proof}
From the expression of the univariate quantile function of the \SC \ distribution \citep*{behboodian2006}, for any directional vector $\vec{u} \in \mathbb{R}^d$, we can write
\begin{eqnarray*}
q_{1-\alpha}(\vec{u}) = u_{1}s(\alpha) + t(\alpha),
\end{eqnarray*}
where $u_{1}$ is the first element of $\vec{u}$. It follows that
\begin{eqnarray*}
Q_{\alpha} &=& \left\{\vec{x} : \vec{u}^\top \vec{x} \leq u_{1}s(\alpha) + t(\alpha), \forall \vec{u}\right\} \\
                 &=& \left\{\vec{x} : u_{1}\frac{(x_{1}-s(\alpha))}{t(\alpha)}+\sum_{i=2}^du_{i}\frac{x_{i}}{t(\alpha)} \leq 1 , \forall \vec{u}\right\}.
\end{eqnarray*}
By observing that the Euclidean unit ball $\{\vec{y}: \vec{y}^\top\vec{y} \leq 1\}$ can be written as $\{\vec{y} : \vec{u}^\top \vec{y} \leq 1, \forall \vec{u}\}$, we conclude that for $\vec{x} \in Q_{\alpha}$, the vectors $\vec{y} = \left[\vec{x}-(s(\alpha),0,\ldots,0)^\top\right]/t(\alpha)$ describe the unit ball and therefore 
\begin{eqnarray*}
Q_{\alpha} = \left\{\vec{x}: (x_{1}-s(\alpha))^2+\sum_{i=2}^d x_{i}^2 \leq t(\alpha)^2\right\}.
\end{eqnarray*}
\end{proof}

\begin{figure}[htbp]
\begin{center}
\includegraphics[scale=0.98]{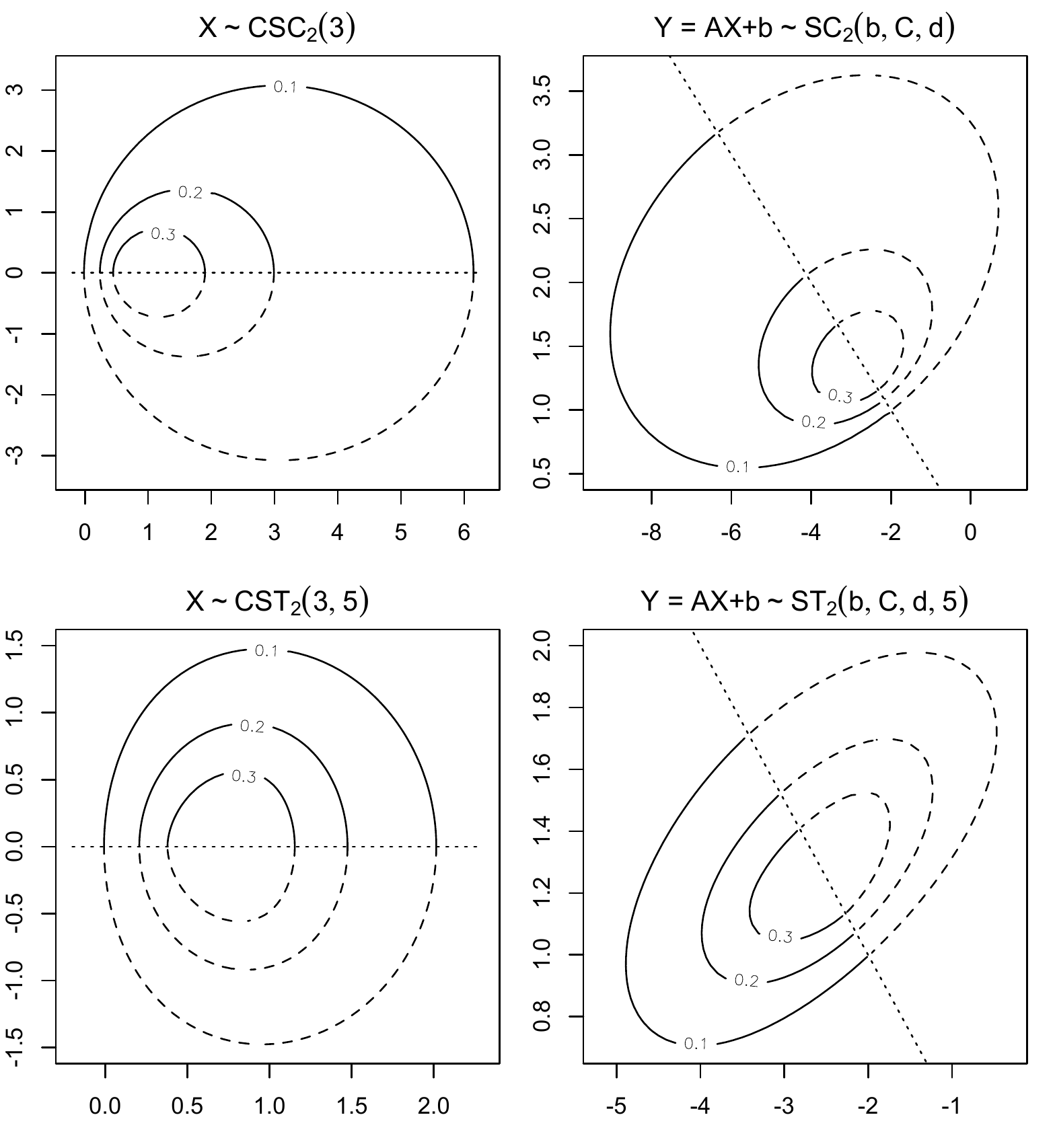}
\caption{Half-space depth contours $\partial Q_{0.1}$, $\partial Q_{0.2}$ and $\partial Q_{0.3}$ in the case of the two bivariate variables $\vec{X}$ (left panels) and $\vec{Y}$ (right panels) as defined in Example \ref{example:st}, with $\nu=1$ (top panels) and $\nu=5$ (lower panels).\label{fig:exampleHD_ST}}
\end{center}
\end{figure}

\begin{figure}[htbp]
\begin{center}
\includegraphics[scale=0.8]{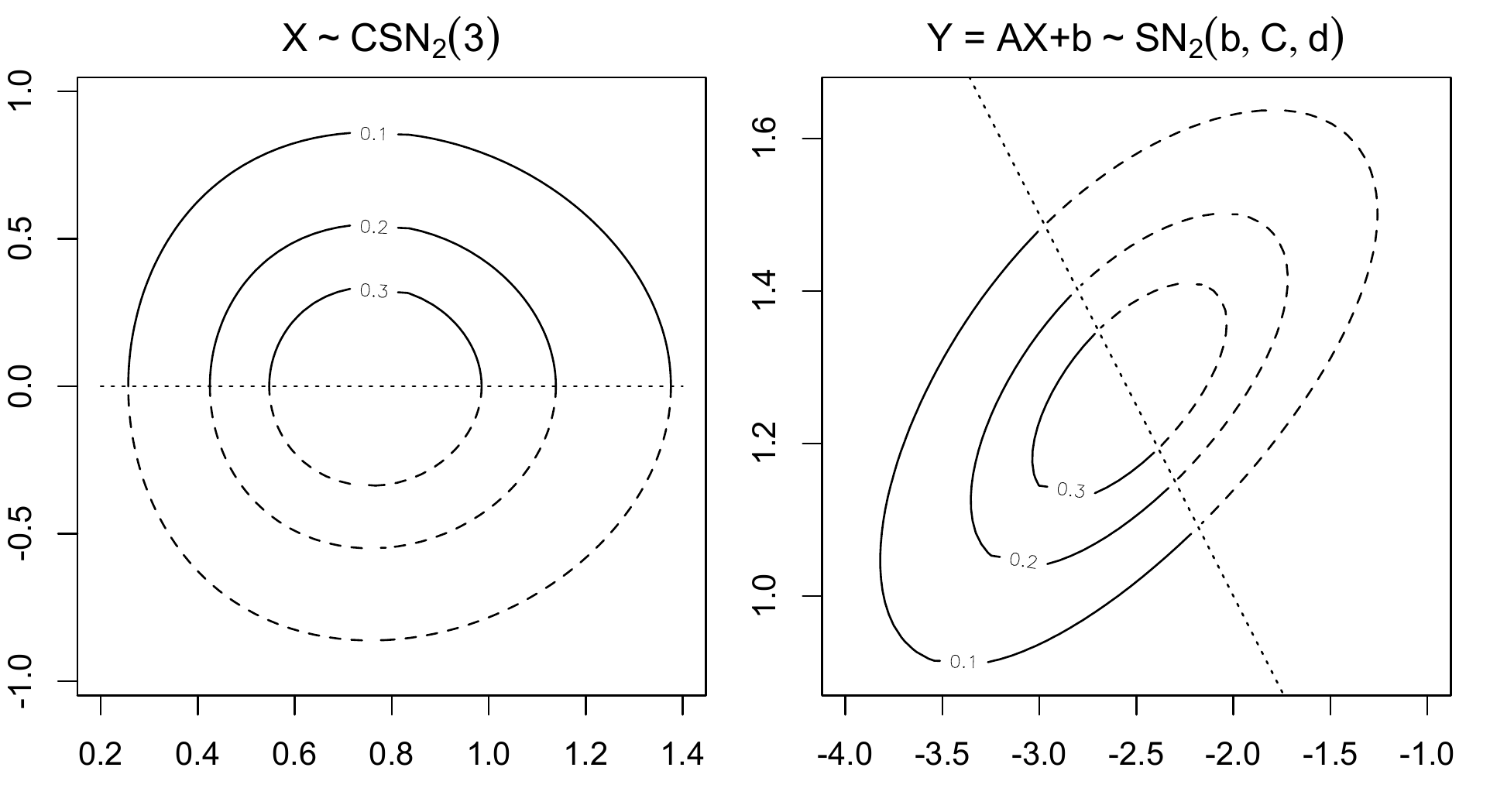}
\caption{Expectile depth contours $\partial E_{0.1}$, $\partial E_{0.2}$ and $\partial E_{0.3}$ in the case of the two bivariate variables $\vec{X}$ (left panel) and $\vec{Y}$ (right panel) as defined in Example \ref{example:sn}. \label{fig:exampleED_SN}}
\end{center}
\end{figure}

We now give some examples to illustrate the depth contours and
expectile depth contours of certain special cases of the ST
distribution. While the skew-Cauchy has elliptical depth contours,
other cases have contours that are near-elliptical; the quality of an
elliptical approximation will be investigated further in
Section~\ref{subsec:prob_misclass}. Algorithm~\ref{algorithm:HDforST}
is used to calculate half-space depth and a similar approach can be
used for expectile depth, as indicated in Example~\ref{example:sn}.
\begin{example}
\label{example:st}
Let $\vec{X} \sim \CST_{2}(3,\nu)$ and define
$$
A = \frac{\sqrt{2}}{2}\left(\begin{matrix}
-1 & -2 \\
1/2 & -1/2
\end{matrix}\right)
$$
and $\vec{b}^\top = (-2, 1)$. The linear transformation $\vec{Y} = A\vec{X}+\vec{b}$ has distribution  $\ST_{2}(\vec{b},C,\vec{d},\nu)$, where
$$
C = \left(\begin{matrix}
5/2 & 1/4 \\
1/4 & 1/4
\end{matrix}\right)
$$
and $\vec{d}^\top \approx (-2.336, 2.828)$. Figure \ref{fig:exampleHD_ST} shows the construction of $\partial Q_{\alpha}$, for $\alpha\in \{0.1, 0.2, 0.3\}$, for the random variables $\vec{X}$ and $\vec{Y}$, letting $\nu$ vary over the values $\{1,5\}$. 
An efficient computation of the depth contours for $\vec{Y}$ is given by applying an affine transformation to the depth contours computed for $\vec{X}$, where skewness is completely absorbed by the first component. Note that from Theorem \ref{theorem:symmhd}, we only need to compute half of the depth contour and obtain the other half by symmetry. In the case of $\nu = 1$ the computation is further simplified by the circular shape of the depth contours of the canonical form.
\end{example}

\begin{example}
\label{example:sn}
Let $\phi(\cdot;a)$ and $\Phi(\cdot;a)$ be the density and distribution functions of a univariate \SN \ distribution with skewness parameter $a$, respectively. If $\vec{X} \sim \CSN_{2}(\gamma)$ and $y = \vec{u}^\top\vec{x}$, then 
\begin{eqnarray*}
\tilde{P}_{\vec{X}}(H_{\vec{u}, \vec{x}}) = \frac{p(y)}{2p(y)+\delta\sqrt{2/\pi}-y}
\end{eqnarray*}
where
\begin{eqnarray*}
p(y) &=& 2\sqrt{\frac{1+\gamma^2}{2\pi}}\Phi\left(\delta y; 0\right)-
\phi\left(\delta y;\tilde{\delta}\right)-y\Phi\left(y; \tilde{\delta}\right), \text{ with }\\
\delta &=& \frac{\gamma}{\sqrt{1+\gamma^2}}\text{ and }\tilde{\delta} = \frac{u_{1}\gamma}{\sqrt{1+\gamma^2(1-u_{1}^2)}}, u_{1} \in (-1,1).
\end{eqnarray*}
Figure \ref{fig:exampleED_SN} shows the expectile depth contours for the random variables $\vec{X}$ and $\vec{Y} = A\vec{X}+\vec{b}$ for $\gamma=3$, where $A$ and $\vec{b}$ are defined in Example \ref{example:st}. 
\end{example}

 \begin{algorithm}
\caption{Computation of \HD \ for the \ST \ distribution \label{algorithm:HDforST}}
Let $\vec{X} \sim \ST_{d}(\boldsymbol{\xi},\Omega,\boldsymbol{\gamma},\nu)$ and $\vec{x} \in \mathbb{R}^d$:
\begin{enumerate}
\item compute $A = P^\top B^{-1}$ where $P$ and $B$ are given in Theorem \ref{theorem:canon_st};
\item compute $\gamma = (\boldsymbol{\gamma}^{\top}\overline{\Omega}\boldsymbol{\gamma})^{1/2}$;
\item transform $\vec{x}$ in $\vec{x}^* = A(\vec{x}-\boldsymbol{\xi})$;
\item Use numerical optimization to minimize 
\begin{eqnarray*}
\min\{F_{Y}(\vec{u}^\top\vec{x}^*), 1-F_{Y}(\vec{u}^\top\vec{x}^*)\}
\end{eqnarray*}
with respect to the directional vector $\vec{u}$, where $Y \sim \CST_{1}(\gamma_*,\nu)$ and $\gamma_*$ given by \eqref{eq:gammastar}.
\end{enumerate}
\end{algorithm}

\subsection{Generalized hyperbolic distribution}
\label{subsec:gh}
A class of multivariate skewed distributions that has received a lot
of attention in the financial literature is the class of generalized
hyperbolic (\GH) distributions;
see~\citet*{mcneil2005} and \citet{bib:eberlein-10}. Let $\boldsymbol{\mu}$,
$\boldsymbol{\kappa} \in \mathbb{R}^d$ denote the parameters of
location and skewness, let $\Sigma \in \mathbf{R}^{d \times d}$ be
a symmetric, positive-definite dispersion matrix and l et
$\lambda \in \mathbb{R}$, $\chi$, $\psi \in \mathbb{R}^+$ be scalars.
$\vec{X}$ has a generalized hyperbolic distribution, written $\vec{X}
\sim \GH_{d}(\boldsymbol{\mu}, \Sigma, \boldsymbol{\kappa},\lambda,
\chi, \psi)$, if it has density 
\begin{displaymath}
f_{\GH_{d}}(\vec{x}) = c \: \frac{K_{\lambda -d/2}\left(\sqrt{(\chi+Q_{\vec{x}})(\psi+\boldsymbol{\kappa}^\top\Sigma^{-1}\boldsymbol{\kappa})}\exp\left\{
\left(\vec{x}-\boldsymbol{\mu}\right)^\top
\Sigma^{-1}
\boldsymbol{\kappa}\right\}\right)}{((\chi+Q_{\vec{x}})(\psi+\boldsymbol{\kappa}^\top\Sigma^{-1}\boldsymbol{\kappa}))^{(d/2-\lambda)/2}},\;\; \vec{x} \in \mathbb{R}^d, % \\
% c&=&\frac{(\chi \psi)^{-\lambda/2} \psi^\lambda (\psi + \boldsymbol{\kappa}^\top\Sigma^{-1}\boldsymbol{\kappa})^{d/2-\lambda}}{(2\pi)^{d/2}|\Sigma|^{1/2}K_{\lambda}(\sqrt{\chi\psi})}\text{ and } Q_{\vec{x}} = (\vec{x}-\boldsymbol{\mu})^\top\Sigma^{-1}(\vec{x}-\boldsymbol{\mu}),
\end{displaymath}
where
\begin{displaymath}
% f_{\GH_{d}}(\vec{x})&=& c \: \frac{K_{\lambda -d/2}\left(\sqrt{(\chi+Q_{\vec{x}})(\psi+\boldsymbol{\kappa}^\top\Sigma^{-1}\boldsymbol{\kappa})}\exp\left\{
% \left(\vec{x}-\boldsymbol{\mu}\right)^\top
% \Sigma^{-1}
% \boldsymbol{\kappa}\right\}\right)}{((\chi+Q_{\vec{x}})(\psi+\boldsymbol{\kappa}^\top\Sigma^{-1}\boldsymbol{\kappa}))^{(d/2-\lambda)/2}}, \vec{x} \in \mathbb{R}^d, \\
c=\frac{(\chi \psi)^{-\lambda/2} \psi^\lambda (\psi + \boldsymbol{\kappa}^\top\Sigma^{-1}\boldsymbol{\kappa})^{d/2-\lambda}}{(2\pi)^{d/2}|\Sigma|^{1/2}K_{\lambda}(\sqrt{\chi\psi})}\text{ and } Q_{\vec{x}} = (\vec{x}-\boldsymbol{\mu})^\top\Sigma^{-1}(\vec{x}-\boldsymbol{\mu}),
\end{displaymath}
and where $K_{\lambda}(\cdot)$ denotes the modified Bessel function of
third kind. This class of distributions can be stochastically
represented as mean-variance mixtures of normal distributions using
the representation
\begin{equation}
\label{eq:gh}
\vec{X} \overset{d}{=} \boldsymbol{\mu}+W\boldsymbol{\kappa} + \sqrt{W}A\vec{Z},
\end{equation}
where
\begin{itemize}
\item [$(i)$] $\vec{Z} \sim N_{d}(\vec{0},I_{d})$;
\item [$(ii)$] $A$ is a $d\times d$ matrix such that $\Sigma = AA^\top$;
\item [$(iii)$] $W$ has a generalized inverse Gaussian (GIG), denoted
  by $W \sim \GIG(\lambda, \chi, \psi)$ , with density function
  (\ref{eq:gig}) in the Appendix. 
\end{itemize}
Note that $\GH_{d}(\boldsymbol{\mu}, a\Sigma,
a\boldsymbol{\kappa},\lambda, \chi/a, a\psi)$ and
$\GH_{d}(\boldsymbol{\mu}, \Sigma, \boldsymbol{\kappa},\lambda, \chi,
\psi)$ are equal in distribution for $a>0$, which causes an
identifiability problem. This problem can be solved by imposing a
constraint on the model parameters; see the NIG case below. \par
An important feature of the \GH \ distribution is its flexibility. It also contains several special cases and, in particular, we consider the following. 
\begin{itemize}
\item [$(1)$]\emph{Normal-inverse-Gaussian} (\NIG) distribution:
  $\lambda=-1/2$ and $\chi =\psi$ (our choice of identifiability
  constraint).
\item [$(2)$]\emph{Skewed-$t$} (\St) distribution: $\lambda = -\nu/2$, $\chi = \nu$ and $\psi = 0$.
\end{itemize}
Other special cases and a more detailed discussion of the \GH \ family of distributions are given in \citet*{mcneil2005}. \par
We now introduce the canonical form of the \GH \ distribution. The
following result follows from the general result on linear
transformations in Appendix~\ref{app:gener-hyperb-distr}.

\begin{theorem}
\label{theorem:canon_gh}
Let $\vec{X} \sim \GH_{d}(\boldsymbol{\mu}, \Sigma,
\boldsymbol{\kappa},\lambda, \chi, \psi)$ where $\Sigma=BB^{\top}$ for $B\in
\mathbb{R}^{d\times d}$. Let $\kappa_{*} =
(\boldsymbol{\kappa}^{\top}\Sigma^{-1}\boldsymbol{\kappa})^{1/2}$ and
define $\vec{X}^* = P^{\top}B^{-1}(\vec{X}-\boldsymbol{\mu})$ where
$P$ is an orthonormal matrix having the first column equal to
$\kappa_{*}^{-1}B^{-1}\boldsymbol{\kappa}$. Then $\vec{X}^* \sim \GH_{d}(\vec{0}, I_{d}, \kappa_{*}\vec{e}_{1},\lambda, \chi, \psi)$.
\end{theorem}
If $\vec{X}$ is in canonical form, then we write $\vec{X} \sim \CGH_{d}(\kappa,\lambda,\chi,\psi)$ where $\kappa$ is the scalar skewness parameter; in the case of the \NIG \ distribution, we write $\vec{X} \sim \CNIG_{d}(\kappa,\psi)$; and in the case of the \St \ distribution, $\vec{X} \sim \CSt_{d}(\kappa,\nu)$. \par
Theorem \ref{theorem:symmhd} in Section \ref{subsec:skewt} can be
easily extended to the \GH \ family using a similar argument that
makes use of the canonical form, thus we omit it. \par

In the following example we calculate half-space depth contours for
special cases of the GH distribution. A similar approach to Algorithm~\ref{algorithm:HDforST}
is used to compute half-space depth at points $\vec{x} \in
\mathbb{R}$. Since the \GH\ family is closed under linear operations,
the probabilities of half spaces are simply computed from the
distribution function of univariate \GH\ distributions. The example
suggests that the depth contours are particularly close to elliptical for
many \GH\ distributions and, in view of this, we also
calculate an approximating
ellipsoid using Algorithm~\ref{algorithm:approx_ellipse_GH}. Expectile
depth is also straightforward to compute, as we illustrate.
\begin{example}
\label{example:hd_ed_gh}
Let $\vec{X} \sim \CSt_{2}(3,\nu)$ and $\vec{Y} \sim \CNIG_{2}(3,\psi)$. Figure \ref{fig:exampleHD_GH} shows some \HD \ contours for $\vec{X}$ (top panels), letting $\nu$ vary over the set $\{3,10\}$, and $Y$ (lower panels), with $\psi \in \{1/10,1\}$. Dotted lines correspond to approximating ellipsoids obtained from Algorithm \ref{algorithm:approx_ellipse_GH}. As shown in Figure \ref{fig:exampleHD_GH} larger values of $\nu$ for $\vec{X}$ and larger values of $\psi$ for $\vec{Y}$ result in a better ellipsoidal approximation of the depth contours (see Section \ref{subsec:relation_rs_eds}). Figure \ref{fig:exampleED_GH} shows \ED \ contours for $\vec{X}$ with $\nu=10$ (left panel) and $\vec{Y}$ with $\psi=1$ (right panel).
\end{example}

\begin{figure}[htbp]
\begin{center}
\includegraphics[scale=0.95]{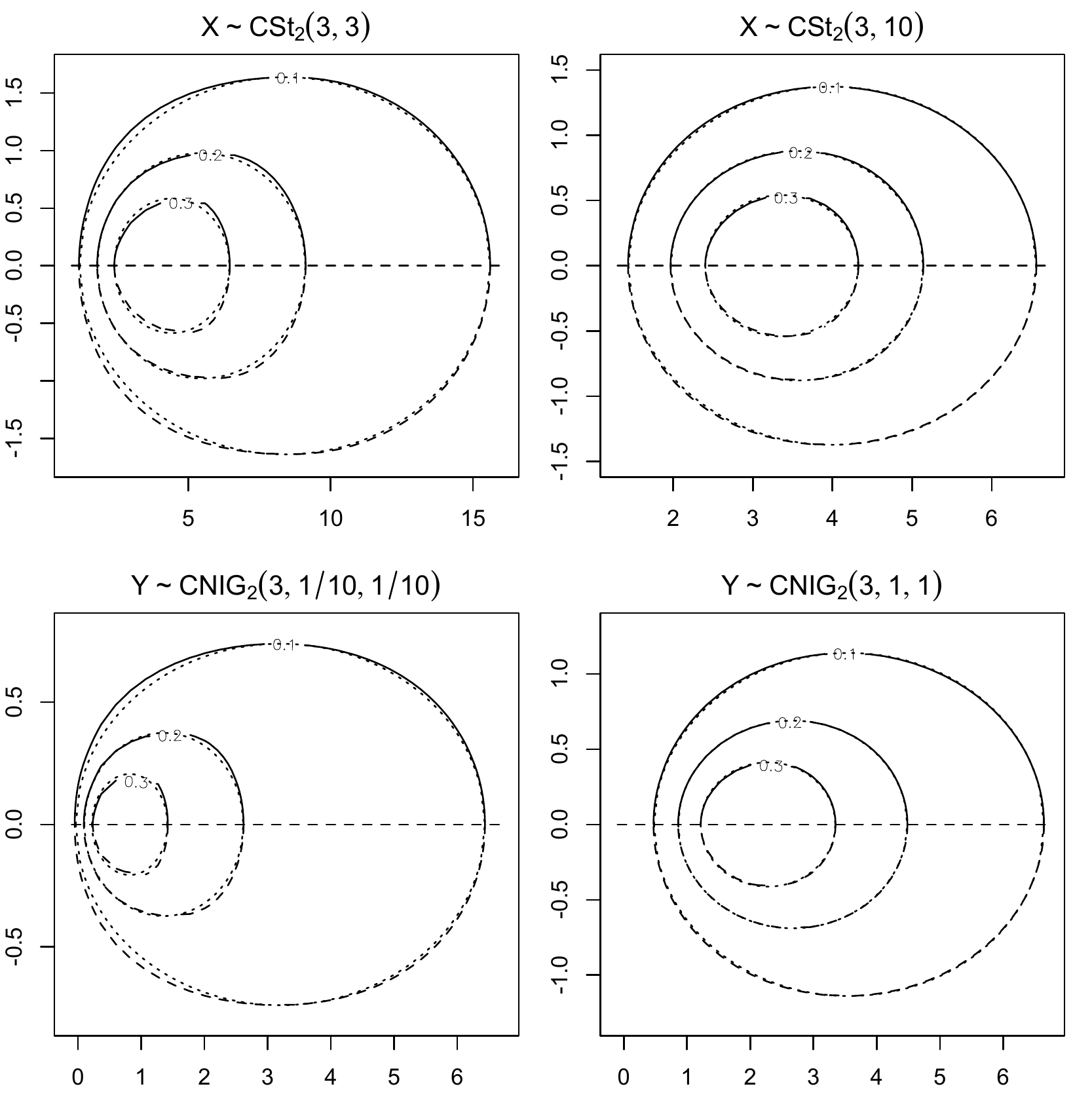}
\caption{Half-space depth contours $\partial Q_{0.1}$, $\partial Q_{0.2}$ and $\partial Q_{0.3}$ in the case of the random variables $\vec{X}$ (top panels), with $\nu = 3$ (left panel) and $\nu = 10$ (right panel), and $\vec{Y}$ (lower panels), as defined in Example \ref{example:hd_ed_gh}; dotted lines correspond to appromximating ellipsoids to each of the depth contours, obtained using Algorithm \ref{algorithm:approx_ellipse_GH}.\label{fig:exampleHD_GH}}
\end{center}
\end{figure}

\begin{figure}[htbp]
\begin{center}
\includegraphics[scale=0.74]{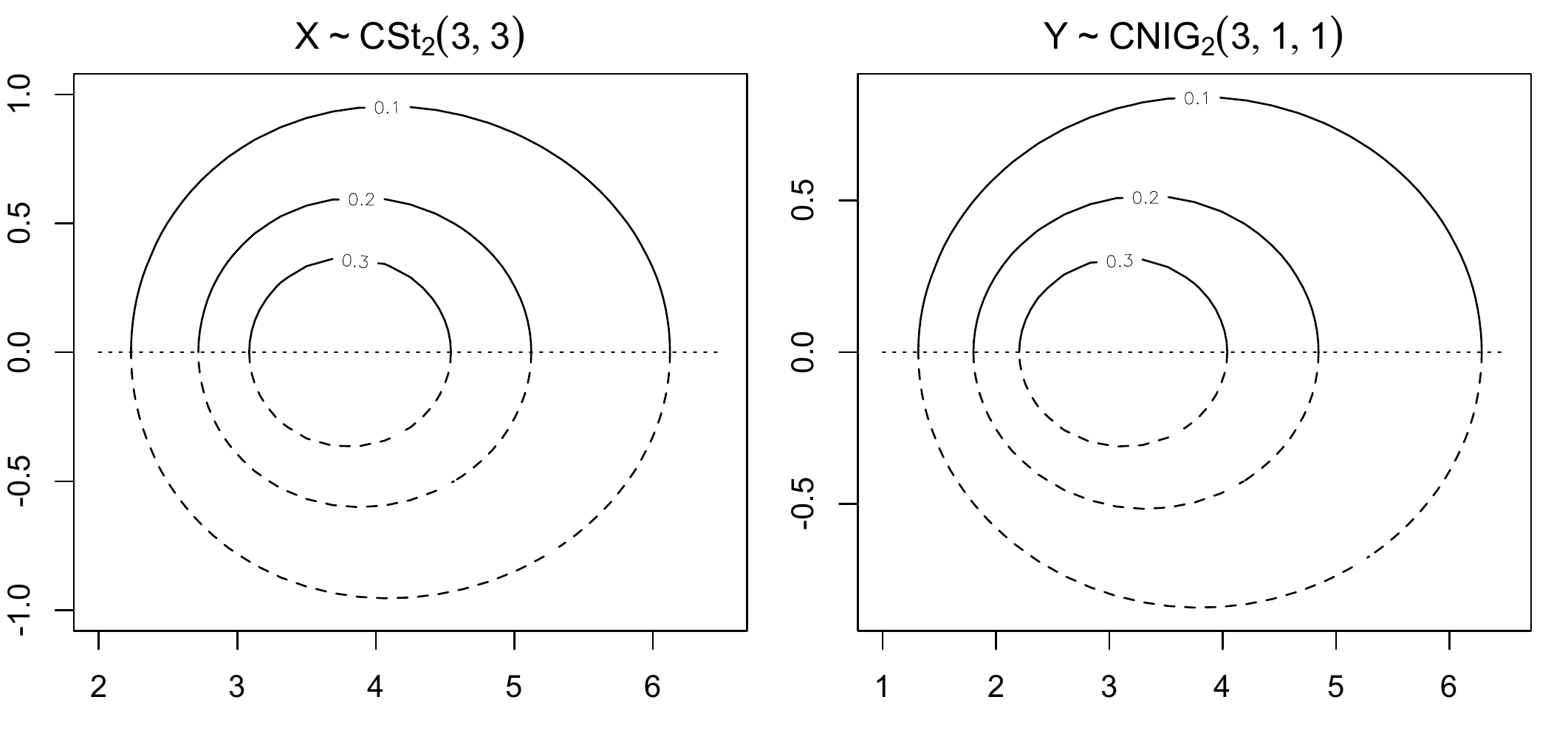}
\caption{Expectile depth contours $\partial E_{0.1}$, $\partial E_{0.2}$ and $\partial E_{0.3}$ for the bivariate variables $\vec{X}$ and $\vec{Y}$ in Example \ref{example:hd_ed_gh}.\label{fig:exampleED_GH}}
\end{center}
\end{figure}

\begin{algorithm}
\caption{Computation of the approximating ellipsoid of $Q_{\alpha}$ for the \GH \ distribution \label{algorithm:approx_ellipse_GH}}
Let $\vec{X} \sim \GH_{d}(\boldsymbol{\mu},\Sigma,\boldsymbol{\kappa},\lambda,\chi,\psi)$:
\begin{enumerate}
\item Compute $\kappa_{*} = (\boldsymbol{\kappa}^{\top}\Sigma^{-1}\boldsymbol{\kappa})^{1/2}$.
\item Compute $A = P^\top B^{-1}$ where $P$ and $B$ are given in Theorem \ref{theorem:canon_gh}.
\item For each component of $\vec{X}^* \sim \CGH_{d}(\kappa_*,\nu)$
  compute the $\alpha$-quantile $a_{i}$, the $(1-\alpha)$-quantile $b_{i}$, and set $c_{i} = (a_{i}+b_{i})/2$ and $d_{i} = |a_{i}-b_{i}|/2$, for $i=1,\ldots,d$.
\item Approximate $Q_{\alpha}$ with the ellipsoid
\begin{eqnarray*}
\tilde{Q}_{\alpha} = \left\{\vec{x} : (\vec{x}-A^{-1}\vec{c}-\boldsymbol{\mu})^\top A^\top D A (\vec{x}-A^{-1}\vec{c}-\boldsymbol{\mu}) \leq 1\right\},
\end{eqnarray*}
where $D = \diag(d_{1}^{-2},\ldots,d_{d}^{-2})$.
\end{enumerate}
\end{algorithm}

\subsection{Relationship between angular symmetry and ellipsoidal depth sets}
\label{subsec:relation_rs_eds}
A natural measure of skewness that quantifies the deviation of a random variable $\vec{X}$ in $\mathbb{R}^d$ from angular symmetry is
\begin{equation}
d_{1}(\vec{X}) = 1/2-\HD_{\vec{X}}(\boldsymbol{\beta}_{\vec{X}}),
\end{equation}
where $\boldsymbol{\beta}_{\vec{X}}$ is the half-space median. The affine equivariance of the median
$\boldsymbol{\beta}_{\vec{X}}$ implies that $d_1(\vec{X})$ is an affine
  invariant measure of skewness. This can be seen by letting $c(\vec{X}) = A\vec{X} + \vec{b}$
  denote the affine transformation that puts $\vec{X}$ into its
  canonical form and observing that
  \begin{eqnarray*}
    d_1(c(\vec{X})) =
  1/2-\HD_{c(\vec{X})}(\boldsymbol{\beta}_{c(\vec{X})})
=  1/2-\HD_{A\vec{X}+\vec{b}}(A \boldsymbol{\beta}_{\vec{X}}+\vec{b})
=  1/2-\HD_{\vec{X}}( \boldsymbol{\beta}_{\vec{X}}) = d_1(\vec{X}).
  \end{eqnarray*}

For the \ST \ and
\GH \ distributions we also define an alternative measure of skewness by
\begin{equation}
d_{2}(\vec{X}) = 1/2-\HD_{c(\vec{X})}(\boldsymbol{\eta}_{c(\vec{X})}),
\end{equation}
where $\boldsymbol \eta_{c(\vec{X})}$ denotes the component-wise median of
$c(\vec{X})$; this measure is simpler to calculate.
Note that
  $d_{2}(\vec{X}) \geq d_{1}(\vec{X})$ since
    $\HD_{c(\vec{X})}(\boldsymbol{\beta}_{c(\vec{X})}) \geq  \HD_{c(\vec{X})}(\boldsymbol{\eta}_{c(\vec{X})})$
and if $\vec{X}$ is angularly symmetric then $d_{1}(\vec{X})=d_{2}(\vec{X})=0$. Alternative measures of multivariate skewness are discussed in a non-parametric context in \citet*{liu1999}. \par
Figure \ref{fig:d2_ST} shows some curves of $d_2$ for different
$\CST_{2}$ distributions against $1/\nu$ and letting $\gamma$ vary
over the set $\{\pm 1, \pm 2, \pm 3, \pm 5, \pm 10, \pm
\infty\}$. Here, we only focused on $\nu \geq 1$ in order to highlight
the different values taken by $d_{2}$ in this interval; however, for
$\nu < 1$, the different curves monotonically increase towards
1/2. Deviation from angular symmetry appears to reach its maximum at
about $0.035$, that corresponds to the case of the independent pair of
variables $(Z_{1}, Z_{2})^\top$, where $Z_{2}$ is a standard normal
variable and $Z_{1}$ is a half-normal variable. However, closeness of
the \ST \ distribution to angular symmetry is indicated for values of
$\nu$ close to 1. Indeed, the following results prove that in the case of
the \SC \ distribution, angular symmetry holds
exactly.
\begin{theorem}
\label{theorem:csc_quantile}
If $(X,Z)^\top \sim \CSC_{2}(\gamma)$, then for any $a \in \mathbb{R}$ the median of $Y_a = X+aZ$ is 
\begin{equation}
\label{eq:csc_median}
\eta = \frac{\gamma}{\sqrt{1+\gamma^2}}.
\end{equation}
\begin{proof}
From the linear forms of the \ST \ distribution (see
Appendix~\ref{subsec:lf_st}) $$Y_a \sim \SC_{1}\left(0,1+a^2,\gamma_{a}\right),$$
where $\gamma_{a}=\gamma/\sqrt{1+a^2(1+\gamma^2)}$. The median of $Y_a$ is then (see Section 2.3 in \citet{behboodian2006})
$$\frac{\sqrt{1+a^2}}{\sqrt{1+\gamma_{a}^2}}\gamma_{a} = \frac{\gamma}{\sqrt{1+\gamma^2}}. $$
\end{proof}
\end{theorem}

\begin{corollary}
\label{corollary:scrad-part1}
If $\vec{X} \sim \CSC_{d}(\gamma)$, then $\vec{X}$ is angularly symmetric about  $\boldsymbol{\eta}^\top=(\eta,0\ldots,0)$ where $\eta$ is given by \eqref{eq:csc_median}.
\end{corollary}
\begin{proof}
It follows from Theorem~\ref{theorem:csc_quantile} that the component-wise median of $\vec{X}$ is $\boldsymbol{\eta}^\top=(\eta,0\ldots,0)$ where $\eta$ is given by \eqref{eq:csc_median}. For any directional vector $\boldsymbol{u} \in \mathbb{R}^d$, we can write
\begin{equation}
\label{eq:pr1}
P_{\vec{X}}\left(H_{\boldsymbol{\eta},\vec{u}}\right) = P\left(u_{1}X_{1} + \sum_{i=2}^{d}u_{i}X_{i} \leq u_{1}\eta\right). 
\end{equation}
If $u_{1} = 0$ the above equation equals $1/2$; let $u_{1} > 0$ and
set $a
= \sqrt{1-u_{1}^2}/u_{1}$. Using Appendix~\ref{subsec:lf_st} it may be
easily shown that $X_{1} + \sum_{i=2}^{d}u_{i}X_{i}/u_{1} \overset{d}{=} Y_a$, where $Y_a$ is defined in Theorem \ref{theorem:csc_quantile}.
It follows that \eqref{eq:pr1} can be expressed as
\begin{eqnarray*}
P\left(X_{1} + \frac{1}{u_{1}}\sum_{i=2}^{d}u_{i}X_{i} \leq
  \eta\right) = P(Y_a \leq \eta) =1/2.
\end{eqnarray*}
Since $u_1$ is arbitrary we conclude that
$\HD_{\vec{X}}(\boldsymbol{\eta})=1/2$ and hence that
$\boldsymbol{\eta}$ is the half-space median, as well as the center of angular symmetry of $\vec{X}$.
\end{proof}

\begin{corollary}
\label{corollary:scrad}
If $\vec{X} \sim \SC_{d}(\boldsymbol{\xi},\Omega,\boldsymbol{\gamma})$, then $\vec{X}$ is angularly symmetric about 
$\boldsymbol{\xi} +\omega\boldsymbol{\delta}$ where
$\omega=\diag(\omega_{11},\ldots,\omega_{dd})^{1/2}$, $\overline{\Omega}=\omega^{-1}\Omega\omega^{-1}$ and
\begin{displaymath}
  \boldsymbol{\delta} = \frac{\overline{\Omega}\boldsymbol{\gamma}}{\sqrt{1+\boldsymbol{\gamma}^\top\overline{\Omega}\boldsymbol{\gamma}}}.
\end{displaymath}
\end{corollary}
\begin{proof} 
Let $\vec{X}^* = c(\vec{X}) =
P^\top B^{-1}(\vec{X}-\boldsymbol{\xi})$ where $P$ and $B$ are defined
in Theorem \ref{theorem:canon_st}. Since $\vec{X}^*$ is
angularly symmetric about its median
$\boldsymbol{\beta}_{\vec{X}^*}=(\eta,0,\ldots,0)$ where $\eta$ is given in~\eqref{eq:csc_median}, it follows easily
that $\vec{X}$ is angularly symmetric about its median
$\boldsymbol{\beta}_{\vec{X}}=
\boldsymbol{\xi}+BP\boldsymbol{\beta}_{\vec{X}^*}=\boldsymbol{\xi} + \omega\boldsymbol{\delta}$.
\end{proof}

\begin{figure}
\begin{center}
\includegraphics[scale=0.9]{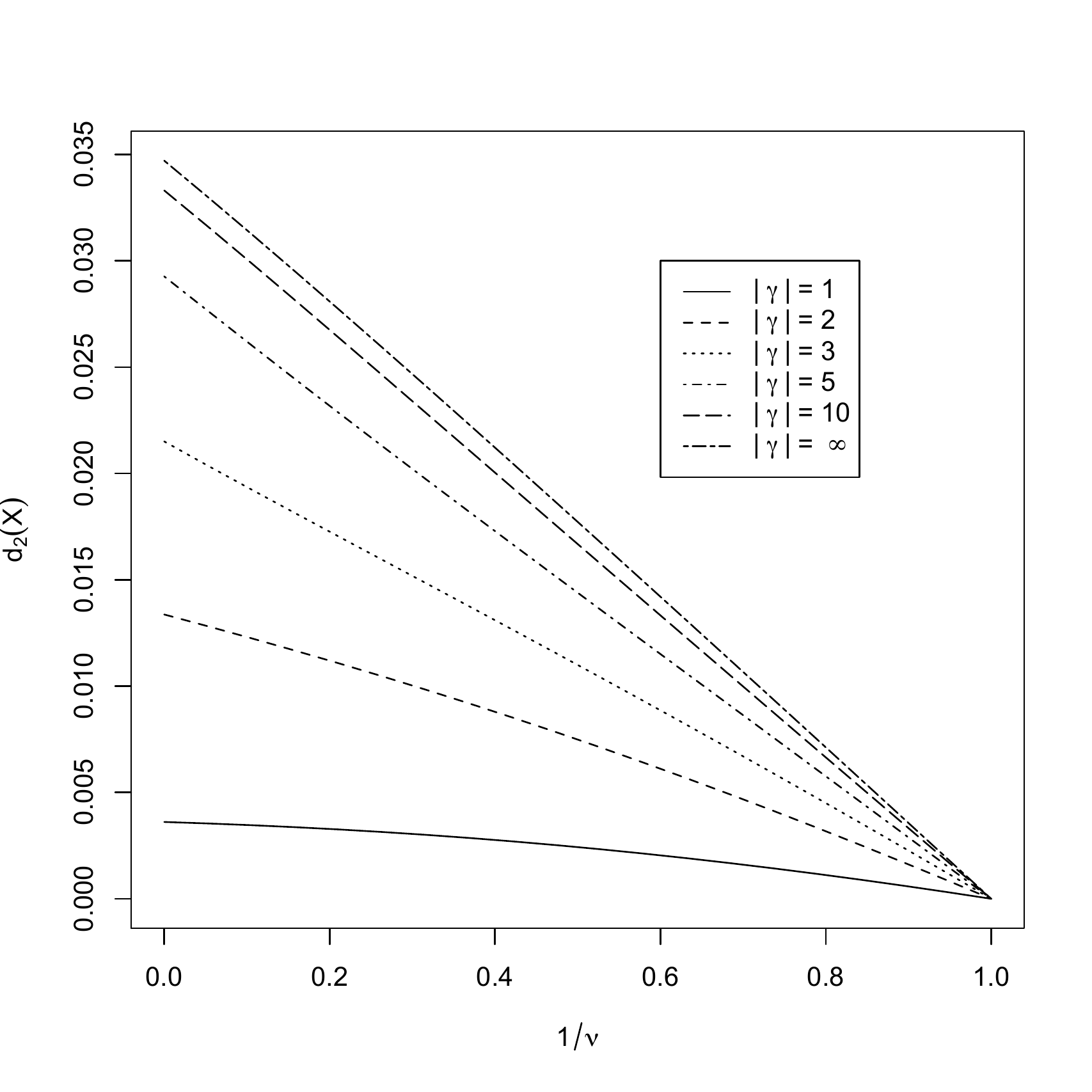}
\caption{Deviation from angular symmetry for the \ST \ family. \label{fig:d2_ST}}
\end{center}
\end{figure}

\begin{figure}
\begin{center}
\includegraphics[scale=0.9]{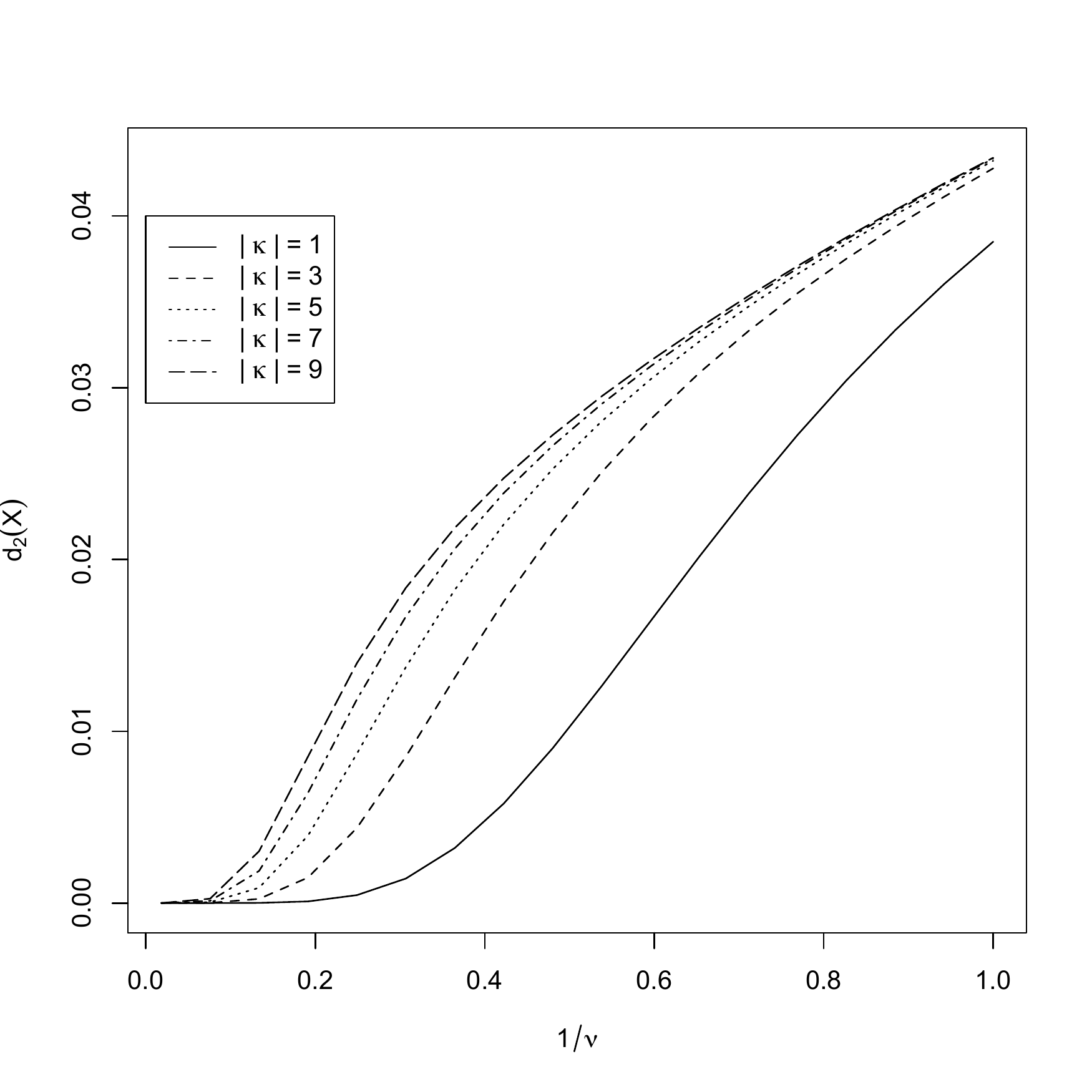}
\caption{Deviation from angular symmetry for the \St \ family. \label{fig:d2_St}}
\end{center}
\end{figure}

\begin{figure}
\begin{center}
\includegraphics[scale=0.9]{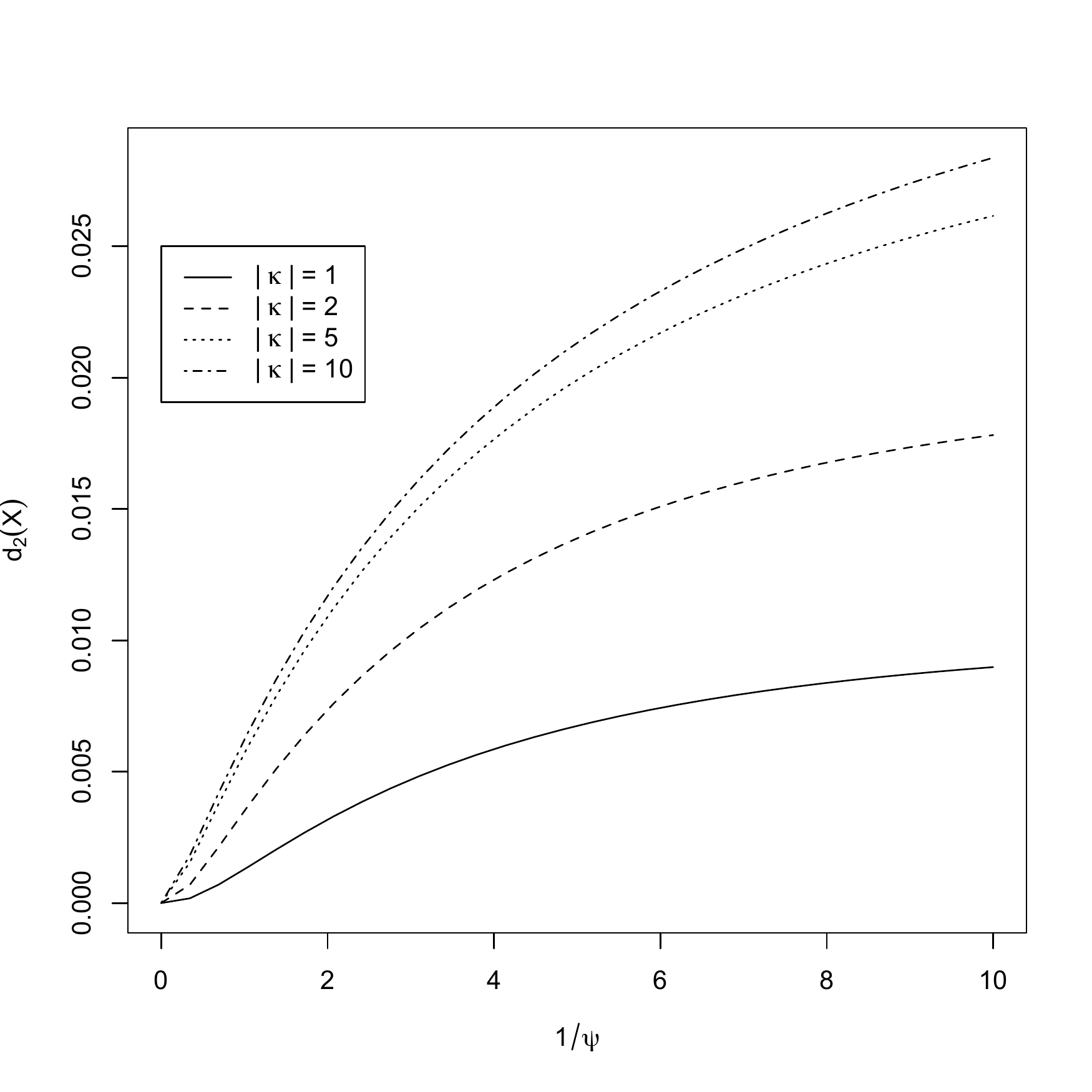}
\caption{Deviation from angular symmetry for the \NIG \ family. \label{fig:d2_nig}}
\end{center}
\end{figure}

Figure \ref{fig:d2_St} is an analogous plot to Figure~\ref{fig:d2_ST}
for the case of the $\CSt_{2}$ distribution. In this case the degrees of freedom play an opposite role with respect to the $\CST_{2}$ case: $d_{2} \rightarrow 0$ as $\nu \rightarrow \infty$. This can be explained as follows. Using the stochastic representation in \eqref{eq:gh}, as $\nu \rightarrow \infty$, $W$ tends to a degenerate distribution which is constant in 1. From \eqref{eq:gh}, we have that $\vec{X}$ tends in distribution to $N_{d}(\boldsymbol{\mu}+\boldsymbol{\kappa}, \Sigma)$ which is elliptically (hence also angularly) symmetric. \par
A similar argument applies to the \NIG \ distribution. As
$\psi \rightarrow \infty$, $W$ tends in distribution to the constant 1
and the same conclusions are drawn as in the previous case. This is
reflected in Figure \ref{fig:d2_nig}, where $d_{2} \rightarrow 0$ for
decreasing values of $1/\psi$. Hence large values of $\psi$ results in a better ellipsoidal approximation.

\subsection{Probability of misclassification using ellipsoidal approximations}
\label{subsec:prob_misclass} 
Let $\vec{X}$ denote a random variable in $\mathbb{R}^d$, belonging either to the \ST \ or \GH \ family. If $\tilde{Q}_{\alpha}$ is the approximating ellipsoid of $Q_{\alpha}$ then the misclassification set is given by
\begin{eqnarray*}
M = M_{1} \cup M_{2},
\end{eqnarray*}
where $M_{1} = \{\vec{x} : \vec{x} \in Q_{\alpha}\text{ and }\vec{x} \notin \tilde{Q}_{\alpha}\}$ is the set of false negatives and $M_{2} = \{\vec{x} : \vec{x} \notin Q_{\alpha}\text{ and }\vec{x} \in \tilde{Q}_{\alpha}\}$ is the set of false positives. The probability of misclassification is
\begin{eqnarray}
\label{eq:prob_misclassification}
P(\vec{X} \in M) = \int_{M} f_{\vec{X}}(\vec{x}) \: d \vec{\vec{x}},
\end{eqnarray}
where $f_{\vec{X}}$ is the density function of $\vec{X}$. The above
integral is intractable, in general, so we use numerical quadrature to approximate the integral as a sum over a fine grid. Letting $\tilde{\vec{x}}_{1},\ldots,\tilde{\vec{x}}_{n} \in M$ be a regular grid covering $M$ with cell area $\Delta$, we have
\begin{eqnarray*}
\int_{M} f_{\vec{X}}(\vec{x}) \: d \vec{\vec{x}} \approx \Delta\sum_{i=1}^n f_{\vec{X}}(\tilde{\vec{x}}_{i}).
\end{eqnarray*}   

For small values of $\alpha$ (for example $\alpha = 0.05$) we find
that, in general, the misclassification probability is relatively low
with values exceeding $0.1$ only in very extreme cases with very
strong asymmetry. Additionally, we find that the main component in the
misclassification probability is generally given by $P(\vec{X} \in M_{1})$ while $P(\vec{X} \in M_{2})$ is often negligible. The next example illustrate these observations.
\begin{example}
\label{example:miscl}
Let $\vec{X} \sim \CNIG_{2}(\kappa,1/10)$, $\vec{Y} \sim \CSN_{2}(\gamma)$ and $\vec{Z} \sim \CSt_{2}(\kappa,5)$. In this example, we consider the approximation of $Q_{0.05}$ with the ellipse $\tilde{Q}_{0.05}$.
In Tables \ref{tab:cnig_miscl}-\ref{tab:cst_miscl} the probability
mass of the sets $M_{1}$ and $M_{2}$, and the index $d_{2}$ are
reported for each of the three random random vectors while the
skewness parameters $\gamma$ and $\kappa$ are allowed to vary. In all three cases, the misclassification probability increases for increasing values of the skewness parameter. The random variable $\vec{X}$ appears to have the highest misclassification probability of about $0.095$ for $\kappa = 30$, a case of extremely high asymmetry.
\begin{table}[ht]
\centering
\caption{Probability of misclassification for $\vec{X} \sim \CNIG_{2}(\kappa,1/10)$. \label{tab:cnig_miscl}}
\begin{tabular}{rccc}
  \hline
  $\kappa$ & $P(\vec{X} \in M_{1})$ & $P(\vec{X} \in M_{2})$ & $d_{2}(\vec{X})$ \\ 
  \hline
 1 & 0.008 & 0.000 & 0.009 \\
 2 & 0.028 & 0.000 & 0.018 \\ 
  5 & 0.069 & 0.000 & 0.026 \\
  15 & 0.089 & 0.000 & 0.029 \\ 
  30 & 0.095 & 0.000 & 0.029 \\
   \hline
\end{tabular}
\caption{Probability of misclassification for $\vec{Y} \sim \CSN_{2}(\gamma)$. \label{tab:csn_miscl}}
\begin{tabular}{rccc}
  \hline
  $\gamma$ & $P(\vec{Y} \in M_{1})$ & $P(\vec{Y} \in M_{2})$ & $d_{2}(\vec{Y})$ \\ 
  \hline
1 & 0.002 & 0.002 & 0.004 \\ 
  2 & 0.008 & 0.005 & 0.013 \\ 
  5 & 0.025 & 0.009 & 0.029 \\ 
  10 & 0.035 & 0.010 & 0.033 \\ 
  50 & 0.036 & 0.010 & 0.035 \\ 
   \hline
\end{tabular}
\caption{Probability of misclassification for $\vec{Z} \sim \CSt_{2}(\kappa,5)$. \label{tab:cst_miscl}}
\begin{tabular}{rccc}
  \hline
  $\kappa$ & $P(\vec{Z} \in M_{1})$ & $P(\vec{Z} \in M_{2})$ & $d_{2}(\vec{Z})$ \\ 
  \hline
1 & 0.001 & 0.000 & 0.005 \\
  3 & 0.002 & 0.001 & 0.013 \\  
  5 & 0.003 & 0.001 & 0.015 \\ 
  10 & 0.003 & 0.001 & 0.017 \\ 
  20 & 0.003 & 0.001 & 0.017 \\ 
   \hline
\end{tabular}
\end{table}
Only in the case of $\vec{Y}$ is the probability mass on $M_{2}$ not
negligible, with a maximum value of about $0.010$ . The random
variable $\vec{Z}$ has the lowest misclassification probability even
for very high values of $\kappa$, reaching a maximum value of about
$0.004$. Note that the measure of skewness $d_{2}$ is also a good
indicator of the quality of the ellipsoidal approximation, although it is not comparable between different families of distributions.
\end{example}
\begin{example}
\label{example:motivating_data}
We now consider the yield data for 3-year and 10-year government bonds, plotted in Figure \ref{fig:motivating_data}. We fitted a bivariate \NIG \ and obtained the following maximum likelihood estimates 
\begin{eqnarray*}
\hat{\boldsymbol{\mu}}^\top &=& (-0.017, -0.016), \\
\hat{\Sigma} &=& \left(
\begin{matrix}
8.860\times 10^{-6} & -5.350\times 10^{-6} \\
-5.350e\times 10^{-6}  & 2.844\times 10^{-5} \\
\end{matrix}\right), \\
\boldsymbol{\kappa}^\top &=& (0.026, 0.020), \\
\hat{\psi} &=& 5.527. \\
\end{eqnarray*}
The resulting index of skewness is $\hat{d}_2=0.001$ which indicates a low level of skewness (as deviation from angular symmetry) of the estimated \NIG \ distribution. Indeed, the misclassification probability of the ellipsoidal approximation is extremely low as is also  evident in Figure \ref{fig:motivating_data}.
\end{example}
\section{Discussion}
\label{sec:discussion}
In this paper we have shown how multivariate scenario sets based on
\HD \ and \ED \ can be efficiently computed in the case of the \ST \
and \GH \ distributions. Computation can be simplified by making use
of a canonical form representation where only one component is
asymmetric at most. Additionally, in the case of multivariate sets
based on \HD , ellipsoids represents a good approximation with a
probability of misclassification exceeding $0.1$ only in cases of
extremely high skewness. We have demonstrated that the quality of the
ellipsoidal approximations can be explained in terms of the closeness
of the \ST \ and \GH \ distributions to angular symmetry. We proposed
a measure of the departure from angular symmetry which is easy to
compute and concordant with the misclassification probability. \par

In our examples of ellipsoidal approximations, we only considered
bivariate cases. However, the affine invariance property and the
availability of a canonical form, with all the asymmetry absorbed in
the marginal distribution of the first component, means that this is
sufficient to gain an understanding of the quality of the
approximation.
 Indeed, the skewness index
$d_{2}$, introduced in Section \ref{subsec:relation_rs_eds}, and the curves shown in
Figures \ref{fig:d2_ST}, \ref{fig:d2_St} and \ref{fig:d2_nig}, are
independent of the dimension of the underlying vector. Although the
probability of misclassification will change with dimension, the
geometry of these distributions means that the changes will remain
modest.\par

 The near-elliptical shape of the depth sets for \ST \ and \GH \
 distributions, and the availability of a simple method of
 constructing an elliptical approximation, makes these distributions
 attractive for modelling the behaviour of financial risk factors in
 the stress testing applications metioned in the Introduction.\par

Although we only considered the \ST \ and \GH \ distributions, we
believe that some of the results can be easily extended to other
multivariate skewed distributions which are closed under affine
transformations and which admit a canonical form representation. An
example is given by skew scale mixtures of normal variates
\citep{branco2001} of which the multivariate skew-slash distribution
\citep{wang2006} is one of the many special cases. \par

Finally, some of the presented material might also be useful to address the related problem of defining multivariate measures of skewed distributions that do not depend on the existence of the moments of the distribution. For example, Theorem \ref{theorem:sc_contours} and Algorithm \ref{algorithm:approx_ellipse_GH} also allow for an exact and approximate computation, respectively, of a multivariate measure of kurtosis proposed by \citet*{wang2005}, who only considered elliptical distributions as parametric examples. Additionally, Corollary \ref{corollary:scrad} provides a measure of location for the \SC \ distribution, namely its center of angular symmetry, that should be preferred to the location parameter $\boldsymbol \xi$ which usually lies in regions far from the ``center'' of the distribution. It would also be interesting to investigate whether the half-space median of the \ST \ and \GH \ distributions is unique or whether a specific point of maximum depth can be identified by making use of the canonical form.

\section*{Acknowledgements}
Some of the work was done while Emanuele Giorgi was a student at the Department of Statistical Sciences in the University of Padua, Italy, under the supervision of Adelchi Azzalini.
\bibliographystyle{biometrika}
\bibliography{biblio,MFE-McNeil}
\appendix
\section{Linear forms}
\label{sec.linear_forms}
Let $A$ be a non-singular $k \times d$ matrix with rank $k \leq d$ and
let $\vec{b} \in \mathbb{R}^k$.
\subsection{Skew-$t$ distribution}
\label{subsec:lf_st}
If $\vec{X} \sim \ST_{d}(\boldsymbol\xi,\Omega, \boldsymbol \gamma, \nu)$, then $\vec{Y} = A\vec{X}+\vec{b} \sim \ST_{k}(\boldsymbol\xi_{\vec{Y}},\Omega_{\vec{Y}}, \boldsymbol{\gamma}_{\vec{Y}},\nu)$, where
\begin{eqnarray*}
\boldsymbol \xi_{\vec{Y}} &=& A\boldsymbol \xi + \vec{b}, \\
\Omega_{\vec{Y}} &=& A\Omega A^\top, \\
\boldsymbol \gamma_{\vec{Y}} &=& \frac{\omega_{\vec{Y}}\Omega_{\vec{Y}}^{-1}C^\top\boldsymbol \gamma}
{\sqrt{1+\boldsymbol \gamma^\top (\overline{\Omega}-C\Omega_{\vec{Y}}^{-1}C^\top) \boldsymbol \gamma}},
\end{eqnarray*}
with $C=\omega^{-1}\Omega A^\top$. \par
If $\vec{X} \sim \CST_{d}(\gamma,\nu)$ and $\vec{u}$ is a directional vector in $\mathbb{R}^d$, then $\vec{u}^\top\vec{X} \sim \CST_{1}(\gamma_*,\nu)$, where 
\begin{equation}
\label{eq:gammastar}
\gamma_* = \frac{u_{1}\gamma}{\sqrt{1+\gamma^2(1-u_{1}^2)}}, u_{1} \in (-1,1).
\end{equation}

\subsection{Generalized hyperbolic distribution}\label{app:gener-hyperb-distr}
If $\vec{X} \sim \GH_{d}(\boldsymbol\mu,\Sigma, \boldsymbol \kappa, \lambda, \chi, \psi)$, then $\vec{Y} = A\vec{X}+\vec{b} \sim \ST_{k}(\boldsymbol\mu_{\vec{Y}},\Sigma_{\vec{Y}}, \boldsymbol{\kappa}_{\vec{Y}},\lambda, \chi, \psi)$, where
\begin{eqnarray*}
\boldsymbol \mu_{\vec{Y}} &=& A\boldsymbol \mu + \vec{b}, \\
\Sigma_{\vec{Y}} &=& A\Sigma A^\top, \\
\boldsymbol \kappa_{\vec{Y}} &=& A\boldsymbol \kappa.
\end{eqnarray*}

If $\vec{X} \sim \CGH_{d}(\kappa, \lambda, \chi, \psi)$ and $\vec{u}$ is a directional vector in $\mathbb{R}^d$, then $\vec{u}^\top\vec{X} \sim \CGH_{1}(u_{1}\kappa,\lambda,\chi,\psi)$.

\section{Generalized Inverse Gaussian distritbuion}
\label{subsec:gig}
If $X$ has a generalized inverse Gaussian (\GIG ) distribution, written as $X \sim GIG(\lambda,\chi,\psi)$, then its density is
\begin{equation}
\label{eq:gig}
f_{\GIG}(x) = \frac{\chi^{-\lambda}(\sqrt{\chi\psi})^\lambda}{2K_{\lambda}(\sqrt{\chi\psi})}x^{\lambda-1}\exp\left\{-\frac{1}{2}(\chi x^{-1}+\psi x)\right\}, x >0,
\end{equation}
where $K_{\lambda}(\cdot)$ is the modified Bessel function of the third kind of order $\lambda$ and with the following constraints on the other parameters: $\chi > 0$, $\psi \geq 0$ if $\lambda < 0$; $\chi > 0$, $\psi > 0$ if $\lambda = 0$; $\chi \geq 0$, $\psi > 0$ if $\lambda > 0$.
\end{document}